%% file: main.tex
\documentclass[a4paper]{article}
\usepackage{a4wide}
\input{commands.tex}
\usepackage{fancyhdr}
\usepackage[bbgreekl]{mathbbol}
\usepackage{authblk}

\author[1]{Florian Oberender and Thorsten Hohage}
\affil[1]{Institut f\"ur Numerische und Angewandte Mathematik, Georg-August Universität Göttingen}
\title{Quantum relative entropy regularization for quantum state tomography%
\footnote{Support from Deutsche Forschungsgemeinschaft (DFG) through  CRC 1456 project 432680300 is gratefully acknowledged.}}

\begin{document}

\maketitle
\begin{abstract}
\noindent
The density matrix is a positive semidefinite operator of trace $1$ characterizing 
the state of a quantum system. 
We consider the inverse problem to reconstruct such density matrices from 
indirect measurements, also known as quantum state tomography. 
To solve such inverse problems in high or infinite dimensional settings, we study 
variational regularization using the quantum relative entropy as 
penalty functional. Quantum relative entropy is an analog of the 
well-known  maximum entropy functional 
with compositions of functions replaced by the spectral functional calculus. 
The main aim of this paper is to establish the regularizing property of this scheme. 
As a crucial intermediate step, we establish lower semi-compactness of the penalty functional with respect to the weak-\(*\)-topology. 
Moreover, we compute the subgradient, proximal operator, and conjugate functional of the quantum relative entropy on finite dimensional spaces. 
This enables us to apply iterative algorithms from convex optimization to 
solve the regularized problems numerically. 
To show the validity and practical value of our results, we apply our theory to the examples of Photon-Induced Near-field Electron Microscopy (PINEM) 
and to optical homodyne tomography.
\\

\\

\noindent%
\end{abstract}


\section{Introduction}\label{sec:introduction}
\input{introduction}
\section{Regularizing property of quantum relative entropy}\label{sec:regularization}
\input{regularization}

\section{Algorithmic approaches}\label{sec:algorithms}
\input{algorithms}

\section{Applications to quantum state reconstruction}\label{sec:phy_prob}
\input{reconstruction}
\section{Conclusions}
\input{conclusions}
\section*{Data Availability Statement}
The code associated with this article is available in `GRO.data', under the reference

\noindent https://doi.org/10.25625/YMGX29.
\appendix
\section{Appendix}\label{sec:appendix}
\input{appendix}
\bibliographystyle{abbrv}
\bibliography{references}
\end{document}

%% file: commands.tex
\usepackage[utf8]{inputenc}
\usepackage{amsmath,amsthm,amssymb,amsfonts}
\usepackage{todonotes}
\usepackage{mathrsfs}
\usepackage{ifthen}
\usepackage{enumitem}
\usepackage{comment}
\usepackage{subcaption}
\usepackage{hyphsubst}
\usepackage{xcolor}
\usepackage{calc}
\usepackage{graphicx}
\usepackage{pgfplots}
\pgfplotsset{compat=1.5.1}
\usepackage{hyperref}
\usepackage[toc,page]{appendix}
\usepackage{cleveref}
\usepackage{leftindex}
\usepackage{algorithm}
\usepackage{algpseudocode}
\usepackage[bbgreekl]{mathbbol}
\DeclareSymbolFontAlphabet{\mathbbm}{bbold}
\DeclareSymbolFontAlphabet{\mathbb}{AMSb}
\usepackage{bbold}

\renewcommand{\Im}{\operatorname{Im}}

\makeatletter
\def\@hspace#1{\begingroup\setlength\dimen@{#1}\hskip\dimen@\endgroup}
\makeatother

\newtheorem{theorem}{Theorem}[section]
\newtheorem{lemma}[theorem]{Lemma}
\crefname{lemma}{lemma}{lemmas}
\Crefname{lemma}{Lemma}{Lemmas}

\crefname{definition}{definition}{definitions}
\Crefname{definition}{Definition}{Definition}
\newtheorem{remark}[theorem]{Remark}
\newtheorem{corollary}[theorem]{Corollary}
\crefname{corollary}{corollary}{corollary}
\Crefname{corollary}{Corollary}{Corollaries}

\crefname{proposition}{proposition}{propositions}
\Crefname{proposition}{Proposition}{Propositions}

\def\gobs{g^{\text{obs}}}
\def\gdag{g^{\dagger}}
\def\predual{^{\#}}
\def\Bop{\mathcal{B}}
\def\Kop{\mathcal{K}}
\def\BopI{\mathcal{B}_{1}}
\def\BopII{\mathcal{B}_{2}}
\def\BopIpsd{\tilde{\mathcal{B}}_{1}^{+}}

\def\Herm{\text{Herm}}
\def\euler{\mathrm{e}}
\def\iu{\mathrm{i}} 

\newcommand{\dd}{\mathrm{d}}

\newcommand{\bpm}{\begin{pmatrix}}
\newcommand{\epm}{\end{pmatrix}}

\DeclareMathOperator{\dom}{dom}

\DeclareMathOperator{\spn}{span}

\DeclareMathOperator{\tr}{tr}
\DeclareMathOperator{\ran}{ran}

\DeclareMathOperator{\KL}{KL}
\DeclareMathOperator{\QKL}{QKL}
\DeclareMathOperator{\argmin}{argmin}
\DeclareMathOperator{\prox}{prox}

\newcommand{\setC}{\mathbb{C}}

\newcommand{\setN}{\mathbb{N}}

\newcommand{\setR}{\mathbb{R}}

\newcommand{\setZ}{\mathbb{Z}}

\newcommand{\calB}{\mathcal{B}}

\newcommand{\calF}{\mathcal{F}}

\newcommand{\calH}{\mathcal{H}}

\newcommand{\calJ}{\mathcal{J}}
\newcommand{\calK}{\mathcal{K}}
\newcommand{\calL}{\mathcal{L}}

\newcommand{\calN}{\mathcal{N}}

\newcommand{\calR}{\mathcal{R}}
\newcommand{\calS}{\mathcal{S}}

\newcommand{\calX}{\mathcal{X}}
\newcommand{\calY}{\mathcal{Y}}

\definecolor{brickred}{rgb}{0.8, 0.25, 0.33}
\definecolor{bostonuniversityred}{rgb}{0.8, 0.0, 0.0}
\definecolor{cornellred}{rgb}{0.7, 0.11, 0.11}
\definecolor{corn}{rgb}{0.98, 0.93, 0.36}
\definecolor{schoolbusyellow}{rgb}{1.0, 0.85, 0.0}
\definecolor{TUblue}{rgb}{0,102,153}
\colorlet{TUbluelight}{TUblue!30!white}

%% file: introduction.tex
Quantum state tomography aims at determining the state of a quantum system by reconstructing its density matrix through a sequence of different measurements. It is a well established technique in quantum optics where heterodyne or homodyne measurements can be used to obtain the density matrix using different reconstruction techniques \cite{Leonhardt:95,DAriano:03}. 
More recently, quantum state tomography for electrons has been established using photon induced near field electron microscopy (PINEM)\cite{Priebe:17}.

Density matrices are positive semidefinite operators $\rho:\cal{H}\to \cal{H}$ of trace $1$ on an 
underlying complex, separable Hilbert space $\cal{H}$.  We will denote the set 
of positive semidefinite operators $\rho:\calH\to\calH$ which are trace class by $\BopIpsd$.  
Only quantities of the form $(e_l,\rho e_l)$ can be measured directly. 
Here $e_l\in\cal{H}$ is an eigenvector of a (self-adjoint)  
operator associated to the observable. These quantities can be seen 
as diagonal elements of a matrix representation of $\rho$ corresponding to the eigen-basis. 
Off-diagonal elements can only be observed indirectly by measuring an evolution of $\rho$. 
The time evolution of a density matrix is described by the von Neumann-Liouville equations, and 
solutions of this equation with initial value $\rho$ are of the form $U_{\theta}\rho U_{\theta}^{*}$ for some unitary operator $U_{\theta}:\cal{H}\to\cal{H}$. 
We assume that $U_{\theta}$ is parameterized by some interaction parameter $\theta$ in a bounded 
subset $I\subset\setR$. This leads to forward operators of the form 
\begin{align}\label{al:general_op}
T:\BopIpsd\rightarrow L^{2}(I,\ell^{p}(\setZ)),\qquad 
    (T\rho)(\theta,l)=\langle e_{l},U_{\theta}\rho U_{\theta}^{*}e_{l}\rangle
\end{align}
where it can sometimes be useful to use operators $U_{\theta}$ of slightly different form. 
The inverse problem of reconstructing a density matrix $\rho$ from the measured data $T\rho$ is 
typically ill-posed \cite{Shi:20,Aubry:08}. 
It poses the additional challenge that estimators should be 
physical density matrices, i.e., positive semi-definite with trace equal to $1$. 
Whereas it is easy to compute the metric projection of a given estimator onto 
this set, such a procedure may not fully exploit the strongly stabilizing effect that 
the positive semidefiniteness and the trace constraints often exhibit. 
For many problem instances, the regularizing effect of these constraints is already 
sufficient, and iterative schemes that ensure these constraints without 
employing any additional regularization, except possibly for early stopping,  
are often successfully used  in practice \cite{Lvovsky:09}. 
The SQUIRRELS algorithm proposed in \cite{Priebe:17} uses quadratic regularization 
by a  Hilbert-Schmidt (or Frobenius) norm together with positive semidefinite programming.  

Instead of the Hilbert-Schmidt norm that was mainly chosen for algorithmic convenience, 
we propose to use the quantum relative entropy as a physically more meaningful penalty functional. Viewing density matrices as a non-commutative version of probability distributions, the quantum relative entropy can be understood as an equivalent of the Kullback-Leibler divergence. 
For operators $\rho,\rho_0:\calH\to\calH$ it is defined by
\begin{align}\label{eq:defi_QKL}
    \QKL(\rho,\rho_{0}):=\begin{cases}
        \infty\quad\text{if }\rho \notin\BopIpsd \text{ or }\rho_{0}\notin\BopIpsd \text{ or }\ker\rho_{0}\nsubseteq\ker\rho\\
        \tr(\rho_{0}-\rho+\rho\ln\rho-\rho\ln\rho_{0})\quad\text{else},
    \end{cases}
\end{align}
see \cite{Lindblad:74}. 
Here \(\rho\ln\rho-\rho\ln\rho_{0}\) is defined to be zero on the kernel of \(\rho\) for \(\ker\rho_{0}\subseteq\ker\rho\). 
Some authors omit  the linear terms in this definition. 
For further details we refer to \cref{sec:prop_QKL}.

In this paper we will consider generalized Tikhonov regularization of the form
\begin{align}\label{al:general_prob}
    \rho_{\alpha} \in \underset{\rho\in\tilde{\BopI}}{\argmin}\,\calS_{\gobs}(T\rho)+\alpha \QKL(\rho,\rho_{0})
\end{align}
with regularization parameter \(\alpha>0\) and some noisy data $\gobs\approx T\rho^{\dagger}$. 
As a data fidelity functional \(\calS_{\gobs}\) we consider either the \(L^{2}\)-norm or the Kullback-Leibler divergence because the data are usually probability densities. 

In \cite{Olivares:07} the quantum relative entropy was already used in a more practical setting to incorporate prior information if the measurement data does not allow for a unique reconstruction. 
In a finite dimensional statistical setting such estimators were analyzed in \cite{koltchinskii:11}, 
and error estimates have been derived.  

We will show, under mild assumptions, that the minimization problem \eqref{al:general_op} has a unique 
solution $\rho_{\alpha}$, and that, for suitable choice of the regularization parameters,
the method is regularizing 
in the sense that $\rho_{\alpha}\to \rho^{\dagger}$ in the trace norm and with respect to $\QKL$ 
as the noise level vanishes.   

\medskip
After repeating the conditions for the regularizing property we summarize many important properties of the quantum relative entropy and show that it is weak-\(*\) lower semicompact. We also characterize the corresponding Bregman divergence, which is again the quantum relative entropy. After showing that \(T\) is weak-\(*\) to weak-\(*\) continuous, we are then able to conclude the regularizing property and show that the sequence of solutions for decreasing noise levels which appropriately chosen regularization parameters converge weak-\(*\) and even in the trace norm. In chapter \ref{sec:algorithms} we move to the finite dimensional setting and derive subgradient, conjugate functional and proximal operators for the quantum relative entropy. With them the applications of iterative solution algorithm such as FISTA and the primal-dual hybrid gradient algorithm are possible. This section is mostly self-contained and provides the final solution algorithms such that they can easily be re-implemented and used in practice. To demonstrate the validity of our theoretical results, we apply them in the setting of homodyne tomography and PINEM and add corresponding numerical examples. 

%% file: regularization.tex
The purpose of this section is to show the use of quantum relative entropy as a penalty functional 
in generalized Tikhonov regularization defines indeed a valid regularization method in the sense 
that the estimation error tends to $0$ with the noise level.  
We begin by recalling classical broadly applicable sufficient conditions for this property. 
In the case of data fidelity terms of norm-power type, the corresponding result dates back to \cite{Seidman:89}. 
More general data fidelity terms were investigated in the PhD thesis  \cite{Poschl:08} 
from which the following theorem is adapted. 

\begin{theorem}\label{the:regularization_general}
Let \(T:\dom T\subset\calX\rightarrow\calY\) be an operator between Banach spaces and let \(\tau_{\calX}\) and \(\tau_{\calY}\) be topologies on \(\calX\) and \(\calY\). Consider an exact solution 
$\rho^{\dagger}\in \dom T$ with corresponding exact data $\gdag:=T\rho^{\dagger}$,  
a sequence of observed data \((\gobs_{k})_{k\in\setN}\) with corresponding data fidelity functionals 
\(\calS_{\gobs_k}:\calY\rightarrow [0,\infty]\) and a penalty functional 
\(\calR : \calX \rightarrow \setR \cup \{\infty\}\) such that the following conditions are satisfied:
\begin{enumerate}[
    leftmargin=*,
    label={(C.\arabic*)},
    ref={(C.\arabic*)}]
    \item \((\calX,\tau_{\calX})\) and \((\calY,\tau_{\calY})\) are real locally convex vector spaces.\label{C1}
    \item  $\calR$ is sequentially lower semicontinuous.\label{C2}
    \item \(\calR\) is sequentially lower semicompact, i.e., $\{\rho\in \calX:\calR(\rho)\leq C\}$ 
    is compact for all $C\in\setR$. \label{C3}
    \item \(\dom T\cap \dom\calR\neq\emptyset\) and \(T\) is sequentially continuous on \(\dom T\cap \dom\calR\).\label{C4}
    \item \(\calS_{\gobs_k}\) is sequentially lower semicontinuous for all data \(\gobs_k\) , \(\calS_{\gobs_k}(g)=0\Leftrightarrow g=\gobs_k\), and  \(\lim_{k\rightarrow\infty}\calS_{\gobs_{k}}(\gdag)=0\). \label{C5}
    \item For the exact data \(\gdag\) there exists a unique \(\calR\)-minimizing solution \(\rho^{\dagger}\in\argmin_{T\rho=g^{\dagger}}\calR(\rho)\).\label{C6}
    \item \(\partial\calR(\rho^{\dagger})\neq\emptyset\).\label{C7}
\end{enumerate}   
Under assumptions \ref{C1}-\ref{C5}  there exists a solution
    \[\rho_{k}\in\argmin_{\rho}\calS_{\gobs_{k}}(T\rho)+\alpha_{k}\calR(\rho) \quad \text{for all} k\in\setN\]
    for any regularization parameter $\alpha_k>0$. If additionally \ref{C6} is satisfied, and 
    the regularization parameters are chosen such that 
    \[\lim_{k\rightarrow\infty}\alpha_{k}=0\text{ and }\lim_{k\rightarrow\infty}\frac{\calS_{\gobs_{k}}(\gdag)}{\alpha_{k}}=0,\]
    then the sequence \(\rho_{k}\) converges to \(\rho^{\dagger}\) in the \(\tau_{\calX}\) topology. Furthermore,
    \[\lim_{k\rightarrow\infty}\calS_{\gdag}(T\rho_{k})=0 \text{ and } \lim_{k\rightarrow\infty}\calR(\rho_{k})=\calR(\rho^{\dagger}).\]
    Finally, if also \ref{C7} holds true, then
    \[\lim_{k\rightarrow\infty}\Delta_{\calR}^{\sigma}(\rho_{k},\rho^{\dagger})=0\quad \text{for }\sigma\in\partial\calR(\rho^{\dagger})\]
    where \(\Delta_{\calR}^{\sigma}\) denotes the Bregman divergence.
\end{theorem}

\subsection{Properties of the quantum relative entropy}\label{sec:prop_QKL}
Due to the different definitions and settings used for the quantum relative entropy in the mathematical physics literature, we first summarize the main properties we need. Our setting includes linear terms and does not assume the operators to act on finite dimensional spaces or be of trace $1$ as it is often done in the physics literature. Most results still transfer to this setting, but some care has to be taken.

Let us first introduce our notation.  As in the introduction, \(\calH\) denotes a complex, separable Hilbert space. Let \(\{f_{l}\}\) denote an arbitrary orthonormal basis of $\calH$. Following \cite[Chapter~VI]{Reed:81} we define the spaces of operators:
\begin{align*}
    \Bop&:=\calL(\calH)&&\text{(bounded operators)}\\
    \Kop&:=\calK(\calH)&&\text{(compact operators)}\\
    \BopII&:\left\{\rho\in\Kop:\sum\nolimits_{l}\|\rho f_{l}\|_{\calH}^{2}<\infty\right\}&&\text{(Hilbert-Schmidt operators)}\\
    \BopI&:=\{\rho\in\BopII:\tr(|\rho|)<\infty\}&&\text{(trace-class operators)}.
\end{align*}
Here $\BopII$ can be shown to be independent of the choice of the basis. 
We use \(\tilde{\Bop},\tilde{\Kop},\tilde{\BopI}\) and \(\tilde{\BopII}\) to denote the corresponding subspaces of hermitian (equivalently, self-adjoint) operators. By the spectral theorem, any $\rho\in \tilde{\Kop}$ 
has an orthonormal basis $\{e_l\}$ of eigenvectors with corresponding eigenvalues $\rho_l$. 
For a real-valued function $\varphi$ on the spectrum of $\rho$, the functional calculus 
$\varphi(\rho)$ is defined by $\varphi(\rho)f:=\sum_l \varphi(\rho_l)(f,e_l)e_l$ 
for all $f$ for which the series converges, and $\varphi(\rho)\in \tilde{\Bop}$ if $f$ is bounded \cite[Theorem~VII.2]{Reed:81}.
This is used both in the definition \eqref{eq:defi_QKL} of the quantum relative entropy 
and in the definition of $\BopI$, where $|\rho|:=\sqrt{\rho^*\rho}$ and 
$\tr(|\rho|):=\sum_l (f_l,|\rho|f_l)$ is again independent of the choice of $\{f_l\}$.

We will choose $\tau_X$ as weak-\(*\)-topology on $\tilde{\Bop}$ in Theorem \ref{the:regularization_general}. 
To properly define this, recall that $\BopI$ can be identified with the dual space of $\Kop$, 
$\BopI\cong \Kop^*$ via the bilinear dual pairing $\langle \rho,\sigma\rangle :=\tr(\rho\sigma)$ 
for $\rho\in \BopI$ and $\sigma\in \Kop^*$,  see, e.g., \cite{Reed:81}.
In other words, $\BopI$ has a predual
\begin{align*}
    \BopI\predual:=\Kop, 
\end{align*}
and the weak topology on $\BopI$ induced by $\Kop$ will be called weak-$*$-topology. 

In Theorem \ref{the:regularization_general} we need real vector spaces, so we cannot use \(\BopI\) directly. However, we can treat \(\tilde{\BopI}\) as a real vector space and characterize the corresponding weak-\(*\)-topology.
\begin{lemma}\label{lem:spaces}
    Treating \(\tilde{\BopI}\) and \(\tilde{\Kop}\) as real vector spaces we have \(\tilde{\BopI}=\tilde{\Kop}^*\), and the weak topology on \(\tilde{\BopI}\) with respect to \(\tilde{\Kop}\) (called 
    weak-$*$-topology on  \(\tilde{\BopI}\) in the following)
    coincides with the relative topolog on $\tilde{\BopI}\subset \BopI$  induced by the 
    weak-\(*\)-topology on \(\BopI\).
\end{lemma}
\begin{proof}
    One can check with a few simple calculations that for every \(\rho\in\tilde{\BopI}\) the functional \(f_{\rho}:\tilde{\Kop}\rightarrow\setR,\sigma\mapsto\tr(\sigma\rho)\) is bounded, \(\setR\)-linear and well-defined. For the other direction, assume that \(f\in\tilde{\Kop}^{*}\), then \(\tilde{f}:\Kop\rightarrow\setC\) defined by
    \[\tilde{f}(\sigma)=f\left(\frac{1}{2}(\sigma+\sigma^{*})\right)-if\left(\frac{i}{2}(\sigma-\sigma^{*})\right)\]
    is well-defined, \(\setC\)-linear and bounded. Then \(\tilde{f}\in\Kop^{*}\) and  there exists \(\rho\in\BopI\) such that \(\tilde{f}(\sigma)=\tr(\sigma\rho)\). Then for all  \(\sigma\in\tilde{\Kop}\) we have
    \[\tr(\sigma\rho)=\tilde{f}(\sigma)=f\left(\frac{1}{2}(\sigma+\sigma^{*})\right)-if\left(\frac{i}{2}(\sigma-\sigma^{*})\right)=f(\sigma)\in\setR.\]
    Now for every unit vector \(v\in\calH\) there exists a projection \(P_{v}\in\tilde{\Kop}\) onto the subspace spanned by \(v\). Then
    \begin{align*}
        0&=2\Im(\tr(P_{v}\rho))=\tr(P_{v}\rho)-\overline{\tr(P_{v}\rho)}=\tr(P_{v}(\rho-\rho^{*}))=\langle v,(\rho-\rho^{*})v\rangle
    \end{align*}
    which implies \(\rho-\rho^{*}=0\) and therefore \(\rho\in\tilde{\BopI}\).

    To see that the topologies coincide, we use a net \((\rho_{n})_{n\in\calN}\subset\tilde{\BopI}\). If it converges with respect to the weak-\(*\)-topology of \(\Kop\), it also converges with respect to the weak-\(*\)-topology of \(\tilde{\Kop}\). For the opposite direction, assume that
\[\lim_{n\in\calN}\tr(\sigma\rho_{n})=\tr(\sigma\rho)\quad\text{for all }\sigma\in\tilde{\Kop}.\]
Then for \(\sigma\in\Kop\)
\begin{align*}
\lim_{n\in\calN}\tr(\sigma\rho_{n})&=\lim_{n\in\calN}\tr\left(\frac{1}{2}(\sigma+\sigma^{*})\rho_{n}\right)+\tr\left(\frac{1}{2}(\sigma-\sigma^{*})\rho_{n}\right)\\
&=\tr\left(\frac{1}{2}(\sigma+\sigma^{*})\rho\right)+\lim_{n\in\calN}(-i)\tr\left(\frac{i}{2}(\sigma-\sigma^{*})\rho_{n}\right)\\
&=\tr\left(\frac{1}{2}(\sigma+\sigma^{*})\rho\right)-i\tr\left(\frac{i}{2}(\sigma-\sigma^{*})\rho\right)=\tr(\sigma\rho).
\end{align*}
So it also converges with respect to the weak-\(*\)-topology coming from \(\Kop\).
\end{proof}
This lemma allows us to work with the usual definition of the weak-\(*\)-topology on \(\tilde{\BopI}\).  Additionally, as the dual space of a Banach space \(\tilde{\BopI}\) is locally convex with respect to the weak-\(*\)-topology,  so \ref{C1} is satisfied.
Regarding continuity, we have the following lemma that was shown recently in \cite{Lami:23}.
\begin{lemma}[\ref{C2},\cite{Lami:23}]\label{lem:QKL_semicont}
The quantum relative entropy \(\QKL\) is weak-\(*\) lower semicontinuous in both arguments.
\end{lemma}

Let us collect some further results from \cite{Lindblad:74}:
\begin{lemma}\label{lem:QKL_basics}
The quantum relative entropy \(\QKL\) defined on trace class operators on a separable Hilbert space has the the following properties.
\begin{enumerate}
    \item \(\QKL(\rho,\sigma)\geq0\) and \(\QKL(\rho,\sigma)=0\Leftrightarrow\rho=\sigma\)
    \item \(\QKL\) is jointly convex.
    \item \(\QKL\) is monotonous, meaning
    \(\QKL(\Phi(\rho),\Phi(\sigma))\leq\QKL(\rho,\sigma)\) for \(\Phi\) a trace preserving expectation (see \cite[Thm. 1]{Lindblad:74} for further details). Especially for a projection \(P\) and \(\Phi(\rho)=P\rho P+(I-P)\rho(I-P)\)
    \[\QKL(P\rho P+(I-P)\rho(I-P),P\sigma P+(I-P)\sigma(I-P))\leq\QKL(\rho,\sigma)\]
\end{enumerate}
\end{lemma}
The quantum relative entropy also satisfies an inequality which relates it to the trace norm in the same way as shown for example in \cite{Eggermont93} for the Kullback-Leibler divergence and the \(\ell^{1}\)-norm.
\begin{lemma}\label{lem:ineq_QKL}
    For \(\rho,\sigma\in\BopIpsd\)
    \[\|\rho-\sigma\|_{\BopI}^{2}\leq\left(\frac{2}{3}\|\rho\|_{\BopI}+\frac{4}{3}\|\sigma\|_{\BopI}\right)\QKL(\rho,\sigma).\]
\end{lemma}
\begin{proof}
    A proof for operators with trace $1$ can be found in \cite{Hiai:81} where it was shown that
    \[\QKL(\rho,\sigma)\geq\frac{1}{2}\|\rho-\sigma\|_{\BopI}^{2}\]
    for \(\tr\rho=\tr\sigma=1\). We will follow this proof and modify the relevant parts to get the more general form. For convenience we define the corresponding functionals \(f_{\rho}:\tilde{\Bop}\rightarrow\setR,a\mapsto\tr(a\rho)\) and \(f_{\sigma}:\tilde{\Bop}\rightarrow\setR,a\mapsto\tr(a\sigma)\). 
    As in \cite{Hiai:81} we can find an orthogonal projection \(P=1_{[0,\infty)}(\sigma-\rho)\) on $\cal{H}$  such that
    \[\|\rho-\sigma\|_{\BopI}=(f_{\sigma}-f_{\rho})(P)-(f_{\sigma}-f_{\rho})(I-P)\leq|(f_{\sigma}-f_{\rho})(P)|+|(f_{\sigma}-f_{\rho})(I-P)|.\]
    It was also shown there that by monotonicity
    \begin{align*}
        \QKL(\rho,\sigma)\geq\left[f_{\sigma}(P)-f_{\rho}(P)+f_{\rho}(P)\ln\frac{f_{\rho}(P)}{f_{\sigma}(P)}\right]+\left[f_{\sigma}(I-P)-f_{\rho}(I-P)+f_{\rho}(I-P)\ln\frac{f_{\rho}(I-P)}{f_{\sigma}(I-P)}\right].
    \end{align*}
    From \cite{Borwein91} we have the inequality
    \[|x-y|\leq\sqrt{\left(\frac{2}{3}x+\frac{4}{3}y\right)\left(y-x+x\ln{\frac{x}{y}}\right)}\quad 
    \text{for all } x\geq0,y>0.\]
    Together with the Cauchy-Schwarz inequality this implies
    \begin{align*}
        (|x-y|+|a-b|)^{2}\leq\left(\frac{2}{3}(x+a)+\frac{4}{3}(y+b)\right)\left(y-x+x\ln{\frac{x}{y}}+b-a+a\ln{\frac{a}{b}}\right)\quad\text{for all } x,a\geq0,y,b>0.
    \end{align*}
    Finally setting \(x=f_{\rho}(P),y=f_{\sigma}(P),a=f_{\rho}(I-P)\) and \(b=f_{\sigma}(I-P)\) gives us the desired inequality.
\end{proof}
We can also bound the trace by \(\QKL\) as follows:
\begin{lemma}\label{lem:QKL_trace_bound}
With Euler's number $\euler$, for all \(\rho,\rho_{0}\in\tilde{\BopI}\) the following inequality holds true:
    \[\QKL(\rho,\rho_{0})\geq\tr(\rho)+(1-\euler)\tr(\rho_{0}).\]
\end{lemma}
\begin{proof}
    If \(\QKL(\rho,\rho_{0})\) is infinite, the inequality is trivial. In particular, we can assume 
    that $\rho,\rho_0\in \BopIpsd$. 
    We can further assume that the discrete spectrum of \(\rho_{0}\) does not contain zero by switching to a subspace if necessary. 
    Then the spectral decomposition of $\rho_0$ has the form \(\rho_{0}=\sum\nolimits_jd_jP_{j}\) with 
    eigenvalues \(d_j>0\) and corresponding eigen-projections \(P_{j}\) satisfying \(\sum\nolimits_jP_{j}=I\). By Jensen's trace inequality (see \cite{Brown:90}) for the convex function \(\eta,[0,\infty)\rightarrow\setR,t\mapsto t\ln t\) we get
    \begin{align*}
        \tr\left(\eta\left(\sum\nolimits_{j}P_{j}\rho P_{j}\right)\right)
        \leq\tr\left(\sum\nolimits_{j}P_{j}\eta(\rho) P_{j}\right)=\sum\nolimits_j\tr(P_{j}\eta(\rho) P_{j})=\sum\nolimits_{j}\tr\left(P_{j}\eta(\rho) )=\tr(\eta(\rho)\right).
    \end{align*}
   Assuming now that the eigenvalues $d_j$ occur with their multiplicities, i.e., $\operatorname{rank}P_j=1$
   for all $j$, we have $P_j\rho P_j = \rho_{jj}P_j$ for some ``diagonal elements'' $\rho_{jj}\geq 0$ of 
   $\rho$. This yields
    \[\sum\nolimits_j\eta(\rho_{jj})\leq\tr(\eta(\rho)).\]
    With this we now get
    \begin{align*}
        \QKL(\rho,\rho_{0})&=\tr(\rho_{0})-\tr(\rho)+\tr(\eta(\rho))-\tr(\rho\ln\rho_{0})\\
        &\geq \sum\nolimits_j\left(d_{j}-\rho_{jj}+\eta(\rho_{jj})\right)-\tr\left(\rho\sum\nolimits_j\ln(d_{j})P_{j}\right)\\
        &=\sum\nolimits_j\left(d_{j}-\rho_{jj}+\rho_{jj}\ln\rho_{jj}-\rho_{jj}\ln d_{j}\right).
    \end{align*}
    The function \(x\mapsto x\ln x-x\ln y-x\) is convex on \([0,\infty)\) and is therefore larger than its linear approximation at \(\euler y\). For \(x\in[0,\infty)\) and \(y\in(0,\infty)\) this implies the inequality
    \begin{align*}
       x\ln x-x\ln y-x&\geq\euler y\ln(\euler y)-(\euler y)\ln y-\euler y+(\ln(\euler y)-\ln y)(x-\euler y)\\
       &=0+1\cdot(x-\euler y)=x-\euler y\,. 
    \end{align*}
    Hence, $d_{j}-\rho_{jj}+\rho_{jj}\ln\rho_{jj}-\rho_{jj}\ln d_{j}\geq d_{j}+\rho_{jj}-\euler d_{j}$,
    and
    \begin{align*}
    \QKL(\rho,\rho_{0})
    &\geq\sum\nolimits_j\left(d_{j}+\rho_{jj}-\euler d_{j}\right)=\tr(\rho)+(1-\euler)\tr(\rho_{0}).\qedhere
    \end{align*}
\end{proof}
Often the quantum relative entropy is just considered for matrices with trace one, so this inequality is not very useful. However in our case, it allows us to show lower semicompactness without using additional trace constraints.

\begin{theorem}[\ref{C3}]\label{the:QKL_compact}
    The functional \(\QKL(\cdot,\rho_{0}):\BopI\rightarrow\setR\cup\{\infty\}\) for \(\rho_{0}\in\BopIpsd\) is sequentially weak-\(*\) lower semicompact. 
\end{theorem}
\begin{proof}
    Let \(C>0\) and let 
    \(\rho\in S_C:=\{\QKL(\cdot,\rho_{0})\leq C\}\). Because \(\QKL(\rho,\rho_{0})\) is finite we have \(\rho\in \BopIpsd\) and \(\tr(\rho)=\|\rho\|_{\BopI}\). From \Cref{lem:QKL_trace_bound} we then get
    \[C\geq\QKL(\rho,\rho_{0})\geq\tr(\rho)+(1-\euler)\tr(\rho_{0})
    \quad \Rightarrow \quad  \|\rho\|_{\BopI}\leq C+(\euler-1)\tr(\rho_{0}).\]
    With this we have shown that $S_C$ is contained in a norm bounded set. 
    As $\Kop$ is the norm closure of the space of finite rank operators it is separable \cite[Theroem~VI.13]{Reed:81}. Then, the sequential version of Banach-Alaoglu's theorem \cite[Theorem~3.17]{Rudin:91} implies that 
    $S_C$ is sequentially weak-$*$ precompact. As $S_C$ is sequentially weak-$*$ closed by \Cref{lem:QKL_semicont}, it is sequentially weak-$*$ compact.
\end{proof}

We can furthermore investigate the corresponding subdifferential and the Bregman divergence. The subdifferential is closely related to Klein's inequality and can be found, e.g.,  in \cite{Lanford:68,Duan:17}. 
We still reproduce the proof as it contains the calculation for the Bregman divergence.

\begin{lemma}\label{lem:QKL_subdiff}
    For \(\rho\in\BopIpsd\) if \(\ln\rho-\ln\rho_{0}\in\tilde{\Bop}\) is well-defined, then 
    \[\ln\rho-\ln\rho_{0}\in\partial(\QKL(\cdot,\rho_{0}))(\rho).\]
\end{lemma}

\begin{proof}
We take \(\sigma\in\tilde{\BopI}\) and calculate
\begin{align*}
    \QKL(\sigma,\rho_{0})&-\QKL(\rho,\rho_{0})\\
    &-\tr((\ln\rho-\ln\rho_{0})(\sigma-\rho))=\tr(\rho_{0}-\sigma+\sigma\ln\sigma-\sigma\ln\rho_{0})\\
    &\phantom{-\tr((\ln\rho-\ln\rho_{0})(\sigma-\rho))=}-\tr(\rho_{0}-\rho+\rho\ln\rho-\rho\ln\rho_{0})\\
    &\phantom{-\tr((\ln\rho-\ln\rho_{0})(\sigma-\rho))=}-\tr(\sigma\ln\rho-\rho\ln\rho-\sigma\ln\rho_{0}+\rho\ln\rho_{0})\\
    &\phantom{-\tr((\ln\rho-\ln\rho_{0})(\sigma-\rho))}=\tr(\rho-\sigma+\sigma\ln\sigma-\sigma\ln\rho)=\QKL(\sigma,\rho)\geq0.
\end{align*}
The last inequality follows from \Cref{lem:QKL_basics}.
\end{proof}
\begin{remark}
    In this context by well-defined we mean that we treat \(\ln\rho-\ln\rho_{0}\) as single operator and not as the difference of two operators, that have to be well-defined. This entails that we set \(\ln\rho-\ln\rho_{0}\) to be \(0\) on the intersection of the kernels of \(\rho\) and \(\rho_{0}\). Furthermore, for \(v\) in a common eigenspace of \(\rho\) and \(\rho_{0}\) corresponding to \(\lambda\) and \(\lambda_{0}\) respectively, we have \((\ln\rho-\ln\rho_{0})v=(\ln\lambda-\ln\lambda_{0})v\) and then accordingly define \(\ln\rho-\ln\rho_{0}\) on the union of the common eigenspaces via linear extension. This allows us to still work with operator logarithms even though the spectra of \(\rho\) and \(\rho_{0}\) always have \(0\) as an accumulation point so \(\ln\rho\) and \(\ln\rho_{0}\) separately are not well-defined bounded operators.
\end{remark}
\begin{remark}[\ref{C7}]
    To fulfill condition \ref{C7} the subdifferential has to be considered as a subset of the dual space of \(\tilde{\BopI}\) with respect to the weak-\(*\) topology. This is not the same as the usual dual space \(\tilde{\Bop}\) and by the properties of the weak-\(*\)  topology \cite[Section~3.14]{Rudin:91} coincides with its predual \(\tilde{\Kop}\) instead. Therefore, the stronger assumption \(\ln\rho-\ln\rho_{0}\in\tilde{\Kop}\) is necessary for \ref{C7}.
\end{remark}
With this subdifferential we can calculate the corresponding Bregman divergence.
\begin{corollary}\label{cor:QKL_bregman_divergence}
    For \(\rho\in\BopIpsd\) if \(\ln\rho-\ln\rho_{0}\in\Bop\) is well-defined, then the Bregman divergence with respect to \(\ln\rho-\ln\rho_{0}\in\partial(\QKL(\cdot,\rho_{0}))(\rho)\) is given by
    \[\Delta_{\QKL(\cdot,\rho_{0})}^{\ln\rho-\ln\rho_{0}}(\sigma,\rho)=\QKL(\sigma,\rho).\]
\end{corollary}
\begin{proof}
    In the proof of the previous theorem we already calculated
    \begin{align*}
    \QKL(\sigma,\rho_{0})&-\QKL(\rho,\rho_{0})-\tr((\ln\rho-\ln\rho_{0})(\sigma-\rho))=\QKL(\sigma,\rho)
\end{align*}
and the left hand side is exactly the Bregman divergence.
\end{proof}

\subsection{Weak-$*$-continuity properties of operator and data fidelity functional}
We write the real sequence spaces as \(\ell^{p}(\setZ)\) and indicate the complex sequence spaces by \(\ell^{p}(\setZ,\setC)\). The codomain of the operator $T$ also has a predual, giving rise to a weak-$*$-topology,
which is given by
\begin{align*}
   L^{2}(I,\ell^{p}(\setZ))\predual&:=\begin{cases}
        L^{2}(I,c_{0}(\setZ)),& p=1\\
        L^{2}(I,\ell^{p'}(\setZ)),& p\in(1,\infty)\quad \frac{1}{p}+\frac{1}{p'}=1.
    \end{cases}
\end{align*}
The facts that $L^{2}(I,c_{0}(\setZ))^* = L^{2}(I,\ell^{1}(\setZ))$ and 
$L^{2}(I,\ell^{p'}(\setZ))^* =  L^{2}(I,\ell^{p}(\setZ))$ 
can be derived from \cite[Chapter 1.3]{Hytonen:16} using the separability of  \(\ell^{p}(\setZ)\) for \(p\in[1,\infty)\). 

 We first show that the operator \(T\) is weak-\(*\) to weak-\(*\) sequentially continuous.

 \begin{lemma}[\ref{C4}]\label{lem:cont_op}
An operator $T$ of the form \eqref{al:general_op}  is weak-\(*\) to weak-\(*\) sequentially  continuous 
if \(\{U_{\theta}\in \Bop:\theta\in I\}\) is uniformly bounded (in particular, if 
all operators $U_{\theta}$ are unitary). 
\end{lemma}
\begin{proof}
    
We start with a sequence \(\rho_{n}\in \BopIpsd\) that converges to \(\rho\) in the 
weak-$*$ topology, i.e., 
\[\lim_{n\rightarrow\infty}\tr(\sigma\rho_{n})=\tr(\sigma\rho)\quad \text{for all } \sigma\in \Kop.\]
We first show point-wise convergence where we fix \(\theta\). For \((a_{l})_{l\in\setZ}\in (\ell^{p'})\) we have
\begin{align*}
    \langle (T\rho)_{\theta},a\rangle&=\sum_{l\in\setZ}(T\rho)(\theta,l)a_{l}=\sum_{l\in\setZ}\langle a_{l}e_{l},U_{\theta}\rho U_{\theta}^{*}e_{l}\rangle=\tr(D_{a}U_{\theta}\rho U_{\theta}^{*})\\
    \text{ with }
    D_{a}v&:=\sum_{l}a_{l}(e_l,v)e_l.
\end{align*}
Then \(D_{a}\) is compact because \(a\in (\ell^{p'})\subset c_{0}\) and
\begin{align*}
    \lim_{n\rightarrow\infty}\tr(D_{a}U_{\theta}\rho_{n} U_{\theta}^{*})&=\lim_{n\rightarrow\infty}\tr(\rho_{n} U_{\theta}^{*}D_{a}U_{\theta})=\tr(\rho U_{\theta}^{*}D_{a}U_{\theta})\\
    &=\tr(D_{a}U_{\theta}\rho U_{\theta}^{*})=\langle (T\rho)_{\theta},a\rangle.
\end{align*}
Because every weak-\(*\)-convergent sequence is norm bounded we can assume that \(\rho_{n}\) is bounded in norm by some constant \(C\). We now take \(a(\theta) \in L^{2}(I,(\ell^{p'}(\setZ)))\) and get
\begin{align*}
    |\langle (T\rho_{n})_{\theta},a(\theta)\rangle|&\leq\|a(\theta)\|_{(\ell^{p'}(\setZ))}\|(T\rho_{n})_{\theta}\|_{\ell_{p}(\setZ)}\leq\|a(\theta)\|_{(\ell^{p'}(\setZ))}\|(T\rho_{n})_{\theta}\|_{\ell_{1}(\setZ)}\\
    &=\|a(\theta)\|_{(\ell^{p'}(\setZ))}\tr(U_{\theta}\rho_{n}U_{\theta}^{*})\leq \|a(\theta)\|_{(\ell^{p'}(\setZ))}\|U_{\theta}^{*}U_{\theta}\|\tr(\rho_{n})\\
    &=\|a(\theta)\|_{(\ell^{p'}(\setZ))}\|U_{\theta}\|^{2}\|\rho_{n}\|_{1}\leq\|a(\theta)\|_{(\ell^{p'}(\setZ))}\|U_{\theta}\|^{2}C.
\end{align*}
So \(|\langle (T\rho_{n})_{\theta},a(\theta)\rangle|\) is point-wisely uniformly bounded by a function in \(L^{2}(I)\) and by the dominated convergence theorem \(\langle (T\rho_{n})_{\theta},a(\theta)\rangle\) converges to \(\langle (T\rho)_{\theta},a(\theta)\rangle\) in \(L^{2}\) which finishes the proof.
\end{proof}

To analyze both $L^2$ and Kullback-Leibler data fidelity terms simultaneously and additionally allow the use of different mixed \(L^{p}\)-norms, we define \(\calS_{\gobs}:L^{2}(I,\ell^{p}(\setZ))\rightarrow[0,\infty]\) by
\[\calS_{\gobs}(f):=\int_{I}\sum_{l\in\setZ}s(\gobs(\theta,l),f(\theta,l))\dd\theta\]
where \(s:\setR\times\setR\rightarrow[0,\infty]\) and \(p\in[1,\infty)\). We get the usual \(L^{2}\)-norm with \(p=2\) and
\[s_{2}(x,y):=\frac{1}{2}|x-y|^{2}.\]
Furthermore by setting
\begin{align*}
    s_{KL}(x,y):=\begin{cases}\infty\quad &\text{if }x<0\text{ or }y<0\text{ or }(x>0\text{ and } y=0)\\
    y\quad&\text{if }x=0\text{ and } y\geq0\\
    y-x+x\ln\frac{x}{y}\quad &\text{else}
\end{cases}
\end{align*}
and \(p=1\) we get the Kullback-Leibler divergence.
As the codomain is a mixed \(L^{p}\)-space, weak-\(*\) lower semicontinuity of the data fidelity functional is not immediately obvious. We will see, that in our case we can reduce the problem to investigating the lower semicontinuity on \(L^{2}(I)\).

\begin{lemma}\label{lem:func_lsc}
    If \(\calF_{s,l}:L^{2}(I)\rightarrow[0,\infty],f\mapsto\int_{I}s(\gobs(\theta,l),f(\theta))\dd \theta\) is convex and lower semicontinuous for all \(l\in\setZ\) with respect to the norm topology, then \(\calS_{\gobs}\) is weak-\(*\) lower semicontinuous. Furthermore, if \(s(x,y)=0\Leftrightarrow x=y\) then \(\calS_{\gobs}(g)=0\Leftrightarrow g=\gobs\).
\end{lemma}
\begin{proof}
We will use that the supremum of lower semicontinuous functionals is lower semi-continuous as well. We then for \(n\in \setN\) define
\[\calF_{n}(f):=\int_{I}\sum_{|l|<n}s(\gobs(\theta,l),f(\theta,l))\dd \theta\]
and because of the non-negativity of the summands we get \(\calF(f)=\sup_{n}\calF_{n}(f)\).
Now consider \((f_{k})_{k}\subset L^{2}(I,\ell^{p}(\setZ))\) which converges in the weak-\(*\)-topology to \(f\). This means
\[\lim_{k\rightarrow\infty}\langle f_{k},a\rangle=\langle f,a\rangle \quad\text{for all }a\in L^{2}(I,(\ell^{p'}(\setZ)))\]
Now let \(P_{n}:L^{2}(I,\ell^{p}(\setZ))\rightarrow L^{2}(I,\setR^{2n-1})\) be the point-wise restriction of \(f_{k}\) such that
\[(P_{n}f)(\theta,l)=f(\theta,l) \quad \text{for all }|l|<n\]
Furthermore, define the embedding \(\iota_{n}:L^{2}(I,\setR^{2n-1})\rightarrow L^{2}(I,(\ell^{p'}(\setZ)))\) by
\[(\iota_{n}a)_{j}(\theta)=\begin{cases}
    a(\theta,l) \quad |l|<n\\
    0 \quad \text{ else}
\end{cases}\]
 We get for all \(a^{(n)}\in L^{2}(I,\setR^{2n-1})\)
\[\lim_{k\rightarrow\infty}\langle P_{n}f_{k},a^{(n)}\rangle=\lim_{k\rightarrow\infty}\langle f_{k},\iota_{n}a^{(n)}\rangle=\langle f,\iota_{n}a^{(n)}\rangle=\langle P_{n}f,a^{(n)}\rangle\]
This implies that \((P_{n}f_{k})_{k}\) converges weakly in \(L^{2}(I,\setR^{2n-1})\). For convex functionals weak lower semicontinuity is equivalent to strong lower semicontinuity. Since \(\calF_{s,l}\) is lower semicontinuous on \(L^{2}(I)\) this transfers to the restriction of \(\calF_{n}\) on \(L^{2}(I,\setR^{2n-1})\), which by definition coincides with \(\calF_{n}\). Therefore, \(\calF_{n}\) is weakly lower semicontinuous and we get
\[\lim_{k\rightarrow\infty}\calF_{n}(f_{k})=\lim_{k\rightarrow\infty}\calF_{n}(P_{n}f_{k})=\calF_{n}(P_{n}f)=\calF_{n}(f).\]
This shows that \(\calF_{n}\) is weak-\(*\) lower semicontinuous and by taking the supremum then also \(\calF\) is weak-\(*\) lower semicontinuous.

For the second part of the statement we use that if \(g=\gobs\) then \(s(\gobs(\theta,l),g(\theta,l))=0\) which implies \(\calS_{\gobs}(g)=0\). For the opposite direction, if \(\calS_{\gobs}(g)=0\) we know from the positivity of \(s\) that \(s(\gobs(\theta,l),g(\theta,l))=0\) almost everywhere. This implies \(g=\gobs\) almost everywhere.
\end{proof}
As a corollary we get the desired properties for \(\KL\) and \(\|\,\|_{2}\).
\begin{corollary}[\ref{C5}]\label{cor:funcs_cont}
    The functionals \(\frac{1}{2}\|\cdot-\gobs\|_{L^{2}(I,\ell^{2}(\setZ))}^{2}\) on \(L^{2}(I,\ell^{2}(\setZ))\) and \(\KL(\gobs,\cdot)\) on \(L^{2}(I,\ell^{1}(\setZ))\) are weak-{*} lower semicontinuous and fulfill \ref{C5}.
\end{corollary}
\begin{proof}
    By \Cref{lem:func_lsc} we just need that \(\|\cdot-f\|_{L^{2}(I)}^{2}\) is lower semicontinuous for all \(f\in L^{2}(I)\) which is obvious and for \(\KL\) we need that \(\KL\) is lower semicontinuous on \(L^{2}(I)\) which follows from the lower semicontinuity of \(\KL\) on \(L^{1}(I)\) that was shown in \cite{Eggermont93}.
\end{proof}

If one is willing to impose further a-priori assumptions on the unknown density matrix $\rho$, then 
weak-$*$-continuity of $T$ can often be established by the following simple lemma:
\begin{lemma}\label{lemm:embedding}
Suppose that a linear operator $T:\dom T\subset \BopI \to \calY$ mapping into some 
Banach space $\calY$  is defnied on a domain $\dom T$ that is weakly sequentially compact in some 
Banach space $\hat{\calB}$, and $\hat{\calB}$ is continuously embedded in $\BopI$ with respect to the norm topologies. Further assume that $T$ is continuous with respect to the norm topologies of $\hat{\calB}$ and 
$\calY$. Then $T$ is weak-$*$ to weak seqeuntially continuous from $\BopI$ to $\calY$. 
\end{lemma}
\begin{proof}
    Consider a sequence $(\rho^{(k)})$ in $\dom T$ converging to $\rho$ with respect to the weak-$*$-topology 
    of $\BopI$ induced by $\Kop$. By weak compactness of $\dom T$, 
     $(\rho^{(k)})$ has a subsequence converging weakly to some $\tilde{\rho}$ in $\hat{\calB}$.  
     Since the embedding $\hat{\calB}\hookrightarrow \BopI$ is also weakly continuous, the subsequence also 
     converges weakly and weak-$*$ to $\tilde{\rho}$ in $\BopI$. Hence $\tilde{\rho}=\rho$, and 
     since every subsequence has a subsequence converging weakly to $\rho$ in $\hat{\calB}$, the whole 
     sequence converges weakly to $\rho$ in $\hat{\calB}$. As strong continuity of $T$ implies weak continuity 
     from $\hat{\calB}$ to $\calY$, it follows that $T\rho^{(k)}\rightharpoonup T\rho$ in $\calY$. 
\end{proof}
Note that in this case it is possible to choose $\tau_{\calY}$ as the weak topology, rather than 
a weak-$*$-topology, which facilitates the verification of sequential lower semicontinuity of $\calS_{\gobs}$.

\subsection{Main result} 
After these preparations we are able to prove the regularizing property for quantum relative entropy. 

\begin{theorem}\label{the:regularity}
    Assume that \(T\) fulfills the assumption of \Cref{lem:cont_op} and that \(\calS_{\gobs}\) is either the \(2\)-norm or the Kullback-Leibler divergence. Furthermore, assume that \ref{C6} is satisfied and that we have sequences \((\gobs_{k})_{k\in\setN}\) of observed data and \((\alpha_{k})_{k\in\setN}\) of parameters that fulfill the same assumptions as in \Cref{the:regularization_general}.
    Then the solutions to the regularized problem
    \[\rho_{k}:=\underset{\rho}{\argmin}\;\calS_{\gobs_{k}}(T\rho)+\alpha_{k}\QKL(\rho,\rho_{0})\]
    exist, converge to \(\rho^{\dagger}\) in the weak-\(*\)-topology and
    \begin{align*}
        \QKL(\rho_{k},\rho_{0})\overset{k\rightarrow\infty}{\rightarrow}\QKL(\rho^{\dagger},\rho_{0})\text{ and }
        \lim_{k\rightarrow\infty}\calS_{\gdag}(T\rho_{k})=0.
    \end{align*}
    If additionally \(\ln\rho^{\dagger}-\ln\rho_{0}\) is well defined and compact, it holds that
    \begin{align*}
        \lim_{k\rightarrow\infty}\QKL(\rho_{k},\rho^{\dagger})=0\text{ and }
        \lim_{k\rightarrow\infty}\|\rho_{k}-\rho^{\dagger}\|_{\BopI}=0.
    \end{align*}
\end{theorem}
\begin{proof}
    The assumption in this theorem together with the previous results in this chapter guarantee that \ref{C1}-\ref{C6} are satisfied and the first three convergence results then follow directly from \Cref{the:regularization_general}. The assumption that \(\ln\rho^{\dagger}-\ln\rho_{0}\) is well-defined and compact additionally gives us \ref{C7} by \Cref{lem:QKL_subdiff}. This property gives us convergence in Bregman divergence which we have established to be \(\QKL\) again in \Cref{cor:QKL_bregman_divergence}. Finally, for the norm convergence we use that \(\rho_{k}\) is weak-\(*\) convergent and therefore norm bounded. By \Cref{lem:ineq_QKL} we then have
    \[\lim_{k\rightarrow\infty}\|\rho_{k}-\rho^{\dagger}\|_{\BopI}^{2}\leq\lim_{k\rightarrow\infty}\left(\frac{2}{3}\|\rho_{k}\|_{\BopI}+\frac{4}{3}\|\rho^{\dagger}\|_{\BopI}\right)\QKL(\rho_{k},\rho^{\dagger})\leq\lim_{k\rightarrow\infty}C\QKL(\rho_{k},\rho^{\dagger})=0. \qedhere
    \]
\end{proof}
\begin{remark}
    The only condition we have not discussed so far is \ref{C6}. It is satisfied for example if \(T\) is injective. Even without it we still at least get all the convergence results for a subsequence of \(\rho_{k}\) to some \(\calR\)-minimizing solution \cite{Poschl:08}.
\end{remark}

%% file: algorithms.tex
In this section we investigate methods to solve the regularized problem. 
To this end, we first define the corresponding discrete problem and then characterize the subdifferential, conjugate functional and proximal operators for the quantum  relative entropy. 
Then the problem can then be solved by standard algorithms from convex optimization.
\subsection{Description of discrete setting}
We discretize the space \(\tilde{\BopI}\) by choosing a suitable finite orthonormal basis \(V:=\{v_{n}\in \calH:1\leq n\leq N\}\) with corresponding orthonormal projection \(P_{V}:\calH\rightarrow \spn V\cong\setC^{N}\). The discrete version of \(\rho\in\tilde{\BopI}\) is then given by \(\hat{\rho}:=P_{V}\rho P_{V}^{*}\)
and can be identified with a Hermitian matrix \(\hat{\rho} \in\Herm(N)\). The actual choice of the basis \(V\) depends on the experimental setup and should be done in a way that the discretization error \(\|\rho^{\dagger}-P_{V}^{*}P_{V}\rho^{\dagger} P_{V}^{*}P_{V}\|_{\BopI}\) is small for the solution. We furthermore assume that \(\hat{\rho}_{0}\) has full rank because \(\QKL(\rho,\rho_{0})\) is infinite for operators with a nullspace that is not included in the nullspace of \(\rho_{0}\). This means that restricting ourselves to a subspace where \(\hat{\rho}_{0}\) has full rank does not reduce our search space.  To apply methods of convex analysis to \(\Herm(N)\) it is advantageous to treat it as a real vector space with inner product given by \(\langle\hat{\sigma},\hat{\rho}\rangle_{\Herm(N)}:=\tr(\hat{\sigma}\hat{\rho})\in\setR\) (cf.\ \Cref{lem:spaces}).

The discretization of the codomain is usually determined by the physical setup which results in an operator 
\(Q:L^{2}(I,\ell^{p}(\setZ))\rightarrow\setR^{M_{\theta}\cdot M_{l}}\) defined by local averaging 
around a finite set of values for \(\theta\) and \(l\).
The discrete forward operator is now given by \(\hat{T}:\Herm(N)\rightarrow\setR^{M_{\theta}\cdot M_{l}}\)
\[\hat{T}(\hat{\rho}):=QT(P_{V}^{*}\hat{\rho}P_{V}).\]
As discussed later for specific applications, it is sometimes useful to use a slightly modified discrete operator instead which mimics properties of the continuous operator in the discrete case. 
In this section we only assume that we have some discrete version of our operator for which we can also compute its adjoint \(\hat{T}^{*}\). Finally, as in the continuous case, we can define \(\widehat{\calS}_{\gobs}:\setR^{M_{\theta}\cdot M_{l}}\rightarrow\setR\cup\{\infty\}\) by
\[\widehat{\calS}_{\gobs}(f):=\sum_{l,\theta}s(\gobs(\theta,l),f(\theta,l))\]
with \(s\) defined as in \Cref{sec:introduction}.
This leads us to the discrete problem
\begin{align}\label{al:prob_disc}
    \hat{\rho}_{\alpha}=\underset{\hat{\rho}\in\Herm(N)}{\argmin}\widehat{\calS}_{\gobs}(\hat{T}\hat{\rho})+\alpha\widehat{\QKL}(\hat{\rho},\hat{\rho}_{0}).
\end{align}

As for the other quantities we use \(\widehat{\QKL}\) to indicate that we work in a finite dimensional space.

\subsection{Convex analysis for the quantum relative entropy on matrix spaces}
We  now investigate further properties of the quantum relative entropy in the finite dimensional setting. Note that all results for the infinite version are also still valid but now we can additionally apply convex analysis on hermitian matrices as developed in \cite{Lewis:96} to exactly compute the subdifferential and the conjugate functional. 
\begin{lemma}\label{lem:QKL_subdiff_disc}
    For \(\hat{\rho}\) in the essential domain of \(\widehat{\QKL}(\cdot,\hat{\rho}_{0})\) the subdifferential is given by
    \begin{align*}
        \partial \left(\widehat{\QKL}(\cdot,\hat{\rho}_{0})\right)(\hat{\rho})=\begin{cases}
            \{\ln \hat{\rho} -\ln\hat{\rho}_{0}\} \quad\text{if } \ker\hat{\rho}=\{0\}\\
            \emptyset \quad\text{else}
        \end{cases}
    \end{align*}
\end{lemma}
\begin{proof}
We first assume that \(\hat{\rho}\) has eigenvalues contained in \((0,\infty)\) so we can compute \(\ln\hat{\rho}\) and \(\ln\hat{\rho}_{0}\). 
Furthermore, the function \(f:(0,\infty)\rightarrow\setR,t\mapsto t\ln t-t\) is convex, and the function \(\tilde{f}:(0,\infty)^{N}\rightarrow\setR,(t_{1},...,t_{N})\mapsto \sum_{j}f(t_{j})\) is convex and does not depend on the order of its arguments. This means that \cite[Theorem 3.2]{Lewis:96} applies, and the corresponding spectral function \(f_{\mathrm{spec}}:\Herm(N)\rightarrow[0,\infty),\hat{\rho}\mapsto\tr(\hat{\rho}\ln\hat{\rho}-\hat{\rho})\) has the subdifferential
\[\partial f_{\mathrm{spec}}(\hat{\rho})=\{f'(\hat{\rho})\}=\{\ln\hat{\rho}\}.\]
To get the subdifferential of \(\widehat{\QKL}\) we use that \(-\tr(\cdot\ln\hat{\rho}_{0})\) is linear and \(\tr(\hat{\rho}_{0})\) is constant. We can then add the corresponding subdifferentials to get
\[\partial \left(\widehat{\QKL}(\cdot,\hat{\rho}_{0})\right)(\hat{\rho})=\partial f_{\mathrm{spec}}(\hat{\rho})+\{-\ln\hat{\rho}_{0}\}+\{0\}=\{\ln\hat{\rho}-\ln\hat{\rho}_{0}\}.\]
Now we just have to show that the subdifferential is empty if the kernel of \(\hat{\rho}\) is nontrivial. Let \(P\) be a projection onto a one dimensional subspace of the nullspace of \(\hat{\rho}\). In contradiction to the statement of the lemma, we now assume that there exists \(\hat{\sigma}\) such that \(-I+\ln\hat{\rho}_{0}+\hat{\sigma}\in\partial (\widehat{\QKL}(\cdot,\hat{\rho}_{0}))(\hat{\rho})\). Then by the definition of the subdifferential we have for all \(t>0\) that
\begin{align*}
&&    \widehat{\QKL}(\hat{\rho}+tP,\hat{\rho}_{0})-\widehat{\QKL}(\hat{\rho},\hat{\rho}_{0})
    &\geq\tr((-I+\ln\hat{\rho}_{0}+\hat{\sigma})(\hat{\rho}+tP-\hat{\rho}))\\
&\Leftrightarrow &    \tr(-tP+(\hat{\rho}+tP)\ln(\hat{\rho}+tP)-\hat{\rho}\ln\hat{\rho}+tP\ln\hat{\rho}_{0})&\geq\tr((-I+\ln\hat{\rho}_{0}+\hat{\sigma})tP)\\
&\Leftrightarrow &     \tr((\hat{\rho}+tP)\ln(\hat{\rho}+tP)-\hat{\rho}\ln\hat{\rho})&\geq t\tr(\hat{\sigma}P)\\
&\Leftrightarrow& \ln t&\geq\tr(\hat{\sigma}P).
\end{align*}
As the last inequality holds for all \(t>0\) and the left hand side tends to \(-\infty\) as 
$t\searrow 0$, we have arrived at a contradiction.
\end{proof}
We can now derive the conjugate functional using Young's equality.
\begin{lemma}\label{lem:QKL_conj}
    For a hermitian matrix \(\sigma\), the conjugate functional of \(\widehat{\QKL}(\cdot,\hat{\rho}_{0})\) is given by
    \[\left(\widehat{\QKL}(\cdot,\hat{\rho}_{0})\right)^{*}(\hat{\sigma})=\tr(\exp(\hat{\sigma}+\ln\hat{\rho}_{0})-\hat{\rho}_{0})\]
\end{lemma}
\begin{proof}
    Take \(\hat{\sigma}\) hermitian and define \(\hat{\nu}:=\exp(\hat{\sigma}+\ln\hat{\rho}_{0})\). 
    Then \(\hat{\nu}\) has full rank, and
    \begin{align*}
        \partial \left(\widehat{\QKL}(\cdot,\hat{\rho}_{0})\right)(\hat{\nu})=\{\ln\hat{\nu}-\ln\hat{\rho_{0}}\}=\{\hat{\sigma}+\ln\hat{\rho}_{0}-\ln\hat{\rho_{0}}\}=\{\hat{\sigma}\}.
    \end{align*}
    Young's equality now yields
    \begin{align*}
        \left(\widehat{\QKL}(\cdot,\hat{\rho}_{0})\right)^{*}(\hat{\sigma})
        &=\tr\left(\hat{\sigma}\hat{\nu}\right)-\widehat{\QKL}(\hat{\nu},\hat{\rho}_{0})\\
        &=\tr\left(\hat{\sigma}\hat{\nu}\right)-\tr\left(\hat{\rho}_{0}-\hat{\nu}+\hat{\nu}(\hat{\sigma}+\ln\hat{\rho}_{0}-\ln\hat{\rho}_{0})\right)\\
        &=\tr\left(\hat{\nu}-\hat{\rho}_{0}\right)
        =\tr\left(\exp(\hat{\sigma}+\ln\hat{\rho}_{0})-\hat{\rho}_{0}\right).\qedhere
    \end{align*}
\end{proof}
With the subdifferential we can also compute the proximal mapping 
\[\prox_{\tau\widehat{\QKL}(\cdot,\rho_{0})}(\hat{\sigma}):=\argmin_{\hat{\rho}}\widehat{\QKL}(\hat{\rho},\hat{\rho}_{0})+\frac{1}{2\tau}\|\hat{\rho}-\hat{\sigma}\|_{2}^{2}.\]
It is essentially the same as for the Kullback-Leibler divergence, and its computation requires the inverse of the bijective function \(g:(0,\infty)\rightarrow\setR,t\mapsto \ln t+t\) which can be expressed using the Lambert \(W\) function as \(g^{-1}(t)=W\left(e^{t}\right)\). Using this expression is numerically unstable for large \(t\), so Newton's method should be used instead.

\begin{lemma}\label{lem:QKL_prox}
For \(\hat{\sigma}\) hermitian
    \[\prox_{\tau\widehat{\QKL}(\cdot,\rho_{0})}(\hat{\sigma})=\tau g^{-1}\left(\frac{\hat{\sigma}}{\tau}+\ln\hat{\rho}_{0}-\ln\tau\right)=\tau W\left(\frac{1}{\tau}\exp\left(\frac{\hat{\sigma}}{\tau}+\ln\hat{\rho}_{0}\right)\right).\]
\end{lemma}
\begin{proof}
    The proximal mapping can alternatively be characterized by
    \[\prox_{\tau\widehat{\QKL}(\cdot,\rho_{0})}(\hat{\sigma})=\hat{\rho}\Leftrightarrow \hat{\sigma}-\hat{\rho}\in\tau\partial\left(\widehat{\QKL}(\cdot,\rho_{0})\right)(\hat{\rho}).\]
    Because the image of \(g^{-1}\) is positive \(\tau g^{-1}(\tau^{-1}\hat{\sigma}+\ln\hat{\rho}_{0}-\ln\tau)=:\tau g^{-1}(\hat{\omega})\) has full rank and we have
    \begin{align*}
        \tau\partial\left(\widehat{\QKL}(\cdot,\rho_{0})\right)(\tau g^{-1}(\hat{\omega}))&=\tau(\ln(\tau g^{-1}(\hat{\omega}))-\ln\hat{\rho}_{0})\\
        &=\tau(\ln\tau+\ln g^{-1}(\hat{\omega})+g^{-1}(\hat{\omega})-g^{-1}(\hat{\omega})-\ln\hat{\rho}_{0})\\
        &=\tau(\ln\tau+\hat{\omega}-g^{-1}(\hat{\omega})-\ln\hat{\rho}_{0})\\
        &=\tau(\tau^{-1}\hat{\sigma}-g^{-1}(\hat{\omega}))=\hat{\sigma}-\tau g^{-1}(\hat{\omega}).
        \qedhere
    \end{align*}
\end{proof}
With the same approach we can also compute the gradient for the conjugate functional, and we can use Moreau's identity to get an expression for the proximal operator of the conjugate functional.
\begin{lemma}\label{lem:QKL_conj_properties}
For \(\hat{\sigma}\) hermitian
    \[\partial\left(\widehat{\QKL}(\cdot,\hat{\rho}_{0})\right)^{*}(\hat{\sigma})=\{\exp(\hat{\sigma}+\ln\hat{\rho}_{0})\}\]
    \text{ and }
    \[\prox_{\tau(\widehat{\QKL}(\cdot,\rho_{0}))^{*}}(\hat{\sigma})=\hat{\sigma}-W\left(\tau\exp(\hat{\sigma}+\ln\hat{\rho}_{0})\right)=\hat{\sigma}-g^{-1}\left(\hat{\sigma}+\ln\hat{\rho}_{0}+\ln\tau\right).\]
\end{lemma}
\begin{proof}
The conjugate functional is just a shifted version of \(\tr(\exp(\cdot))\) which by the previously mentioned theory has the subdifferential \(\{\exp(\cdot)\}\). Inserting the shift, we get the claimed formula. For the proximal operator of the conjugate functional we can use Moreau's identity and get
    \begin{align*}
        \prox_{\tau(\widehat{\QKL}(\cdot,\rho_{0}))^{*}}(\hat{\sigma})&=\hat{\sigma}-\tau\prox_{\tau^{-1}\widehat{\QKL}(\cdot,\rho_{0})}(\tau^{-1}\hat{\sigma})\\
        &=\hat{\sigma}-\tau \tau^{-1}W\left(\tau\exp\left(\tau\frac{\hat{\sigma}}{\tau}+\ln\hat{\rho}_{0}\right)\right)\\
        &=\hat{\sigma}-W\left(\tau\exp(\hat{\sigma}+\ln\hat{\rho}_{0})\right).\qedhere
    \end{align*}
\end{proof}
Now that we have computed theses quantities for the quantum relative entropy, we can use standard algorithms from convex optimization to solve the regularized problem. The main downside of this approach is that many calculations require the computation of an eigenvalue decomposition to get the matrix logarithm or apply the function \(g^{-1}\) on the spectrum. 
Fortunately, the desired density matrices are often not very large, so these calculations are still feasible.

Many iterative algorithms can be accelerated significantly if one of the functionals that is minimized is strictly convex with some known convexity parameter \(\mu\). While \(\widehat{\QKL}\) itself is not strictly convex it is strictly convex if the eigenvalues of the matrix-valued argument are bounded from above. As we usually try to find density matrices, the eigenvalues of which are bounded by $1$, the following lemma can be useful:
\begin{lemma}\label{lem:QKL_conv_param}
    \(\hat{QKL}(\cdot,\hat{\rho}_{0})\) is strictly convex with parameter \(\mu\) on the set
    \[M_{\mu}:=\{\hat{\rho} \in \Herm(N):\text{eigenvalues of }\hat{\rho} \text{ not bigger than }\mu^{-1}\}.\]
\end{lemma}
\begin{proof}
    We have to show that \(\hat{\rho}\mapsto\hat{QKL}(\hat{\rho},\hat{\rho}_{0})-\frac{\mu}{2}\tr(\hat{\rho}^{2})\) is convex on the claimed set. For this we can ignore the terms that are constant or linear in \(\hat{\rho}\). Hence, it suffices to show that \(\hat{\rho}\mapsto\tr(\hat{\rho}\ln\hat{\rho}-\frac{\mu}{2}\hat{\rho}^{2})\) is convex on \(M_{\mu}\). 
    As this function only depends on the eigenvalues of \(\hat{\rho}\) and is independent of their order, this is the case if and only if the corresponding real-valued function \(x\mapsto x\ln x-\frac{\mu}{2}x^{2}\) is convex (see \cite{Davis:57}). Computing the second derivative shows that this is the case for \(x\leq\frac{1}{\mu}\).
\end{proof}

\subsection{Iterative solution algorithms}
We will now present some convex optimization algorithms that are able to efficiently and reliably compute the solution of the discrete regularized problem \eqref{al:prob_disc}. 
We begin with the simpler case where the squared \(L^{2}\)-norm is used as a data fidelity term. 
In this case we can use a version of FISTA proposed in \cite{Beck:09}, which had already been proposed for this problem without regularization in \cite{Bolduc:17}. 
As usual for forward backward splitting algorithms, we alternate between a gradient step for the data fidelity functional and an application of the proximal operator. During the iteration the step size is changed to accelerate the convergence. As we usually reconstruct density matrices which have eigenvalues not bigger than $1$, we can in practice assume \(\widehat{\QKL}\) to be strictly convex with parameter \(\mu\) and use the faster versions of the corresponding algorithm  which converge exponentially \cite{ChambollePock:16}. In the settings we considered, choosing \(\mu=0.5\), which by \Cref{lem:QKL_conv_param} corresponds to eigenvalues smaller than \(2\), works well. 

\begin{algorithm}
\caption{Generalized accelerated FISTA}\label{alg:gen_fista}
\begin{algorithmic}
\Require \(\mu > 0;\,\tau\in(0,\|T^{*}T\|^{-1});\,\rho^{(0)};\)
\State\(t_{0} := 0;\,\rho^{(-1)}:= \rho^{(0)};\,q := \frac{\tau\mu}{1+\tau\mu};\)
\For{\(l=0,1,2,...\)}
\State \(t_{l+1}=\frac{1}{2}\left(1-qt_{l}^{2}+\sqrt{(1-qt_{l})^2+4t_{l}^{2}}\right)\)
\State \(\beta_{l}=\frac{t_{l}-1}{t_{l+1}}(1+(1-t_{l+1})\tau\mu)\)
\State \(\tilde{\rho}^{(l)}=\rho^{(l)}+\beta_{l}(\rho^{(l)}-\rho^{(l-1)})\)
\State \(\sigma_{l}=\tilde{\rho}^{(l)}-\tau T^{*}(T\tilde{\rho}^{(l)}-\gobs)\)\Comment{Gradient step}
\State \(\rho^{(l+1)}=\prox_{\alpha\tau\QKL(\cdot,\rho_{0})}(\sigma_{l})\)\Comment{Proximal step}
\EndFor
\end{algorithmic}
\end{algorithm}

If we use the Kullback-Leibler divergence, forward backward splitting algorithms often become numerically unstable because taking a gradient step for the data fidelity functional can be impossible if the iterate is not in the essential domain or the gradient can get extremely large for values close to zero. 
As an alternative, one can use primal dual hybrid gradient methods, like the Chambolle-Pock algorithm \cite{Chambolle:11}. Here we alternate between a proximal step for the conjugate of our data fidelity functional and a proximal step of our penalty functional. By assuming strict convexity, we can again optimize our parameter choices and significantly accelerate the convergence such that it is quadratic \cite{ChambollePock:16}.

\begin{algorithm}
\caption{Accelerated Chambolle-Pock algorithm}\label{alg:gen_cp}
\begin{algorithmic}
\Require \(\mu > 0;\,\alpha > 0;\,\tau_{0}>0;\,\nu_{0}>0;\,\tau_{0}\nu_{0}\|T^{*}T\|\leq1;\,\rho^{(0)};\,g^{(0)};\)
\State\(\tilde{\rho}^{(0)}=\rho^{(0)}\)
\For{\(l=0,1,2,...\)}
\State \(\tilde{g}^{(l)}=g^{(l)}+\nu_{l}T\tilde{\rho}^{(l)}\)
\State \(g^{(l+1)}=\frac{1+\tilde{g}^{(l)}}{2}-\sqrt{\left(\frac{1+\tilde{g}^{(l)}}{2}\right)^{2}-\tilde{g}^{(l)}+\nu_{l}\gobs}\)\Comment{Prox of conjugate data fidelity functional}
\State \(\rho^{(l+1)}=\prox_{\tau_{l}\alpha\QKL(\cdot,\rho_{0})}(\rho^{(l)}-\tau_{l}T^{*}g^{l+1})\)\Comment{Proximal step for penalty functional}
\State \(\beta_{l}=(1+2\mu\tau_{l})^{-\frac{1}{2}};\,\tau_{l+1}=\beta_{l}\tau_{l};\,\nu_{l+1}=\frac{\nu_{l}}{\beta_{l}};\)
\State \(\tilde{\rho}^{(l+1)}=\rho^{(l+1)}+\beta_{l}(\rho^{(l+1)}-\rho^{(l)})\)
\EndFor
\end{algorithmic}
\end{algorithm}
So far we have not mentioned when to stop the algorithms. 
One of the simplest of multiple different possibilities is to check the change between consecutive iterates in some norm and to stop when this change becomes smaller than a predefined threshold. 
A more sophisticated approach that additionally provides an estimate of the distance between the current iterate to the exact minimizer, relies on the duality gap. 
\begin{theorem}[Duality gap]\label{the:duality_gap}
    Let \(\hat{\rho}_{\alpha}\) be the solution to problem \eqref{al:prob_disc} with full rank, then for all \(\hat{\rho}\in\dom\hat{T}\) and \(\hat{g}\in\ran\hat{T}\) we have 
    \[
    \widehat{\QKL}(\hat{\rho},\hat{\rho}_{\alpha})\leq\frac{1}{\alpha}\widehat{\calS}_{\gobs}(\hat{T}\hat{\rho})+\widehat{\QKL}(\hat{\rho},\hat{\rho}_{0})+\frac{1}{\alpha}\widehat{\calS}_{\gobs}^{*}(-\alpha \hat{g})+\left(\widehat{\QKL}(\cdot,\hat{\rho}_{0})\right)^{*}(\hat{T}^{*}\hat{g}).\]
\end{theorem}
A proof of this statement, which is based on Rockafellar-Fenchel duality \cite{Rockafellar:67,Zeidler:13}, is included in the appendix. The rank assumption is needed to guarantee the existence of subdifferentials in line with \Cref{lem:QKL_subdiff_disc}.

The right hand side of this inequality, the duality gap, can be evaluated without knowing the solution \(\hat{\rho}_{\alpha}\) and can therefore be used to obtain an upper bound on the distance to the current iterate \(\hat{\rho}\). To make this bound as tight as possible, we choose \(\hat{g}\in-\frac{1}{\alpha}\partial\widehat{\calS}_{\gobs}(T\hat{\rho})\). We then stop the iterations if the duality gap and hence the error 
\(\widehat{\QKL}(\hat{\rho},\hat{\rho}_{\alpha})\)
is smaller than a predetermined threshold. The possibility to control the distance to the solution is a unique advantage of this approach compared to other iterative solution methods used in quantum tomography so far.

While both presented algorithms are numerically relatively stable, due to inaccuracies during the eigenvalue decomposition it can still happen that eigenvalues are slightly smaller than zero which leads to problems during the computation of the duality gap. For this reason we add a small offset of about \(10^{-10}\) to the eigenvalues before doing further computations.

%% file: reconstruction.tex
We will now consider two different applications where quantum state reconstruction is performed. There we will show how the abstract theory can be applied and present numerical examples that demonstrate the results of our method in practice. The code for all experiments is available at \cite{Oberender:26qkldata}.

\subsection{Electron tomography by PINEM}
In \cite{Priebe:17}, it is was demonstrated that photon induced near field electron microscopy (PINEM) 
enables the reconstruction of the quantum state of a beam of free electrons by using the interaction of electrons with laser photons of different frequencies, while precisely controlling their relative phase. 
For the reconstruction of the density matrix from the data a new method named SQUIRRELS (Spectral Quantum Interference for the Regularized Reconstruction of free ELectron States) has been proposed, which relies on Hilbert-Schmidt norm regularization and also incorporates positive semi definiteness- and trace constraints. Additionally Bayesian methods, maximum likelihood techniques and neural networks have been proposed and compared in \cite{Jeng:25}.

By using quantum relative entropy regularization instead, further constraints are not necessary, 
and we can incorporate further prior knowledge into our reconstruction by the choice of \(\rho_{0}\). 
We will consider the same experimental setup that has been analyzed in \cite{Shi:20}. 
The forward operator \(T:\tilde{\BopI}\rightarrow L^{2}([-\pi,\pi),\ell^{1}(\setZ))\) in this 
reference is given by
\begin{align*}
    (T\rho)(\theta,l)&=\langle e_{l},U_{\theta}\rho U_{\theta}^{*}e_{l}\rangle
    \quad \text{ with } U_{\theta}e_{l}=\sum_{k\in\setZ}\euler^{\iu (k-l)\theta}J_{k-l}(2|g|)e_{k}
\end{align*}
where $\{e_l\}$ is the standard basis of $\ell^2(\setZ)$. 
The operators \(U_{\theta}\) are unitary, so \Cref{lem:cont_op} applies, 
and the same holds true for other PINEM based forward operators where the operators \(U_{\theta}\) are slightly different. For this particular example, it has already been shown in \cite{Shi:20} that \(T\) is injective. 
This means that \Cref{the:regularity} applies directly and yiedls the regularizing property of quantum relative entropy regularization for this example.

For the numerical experiments we use the same general setup as described in the supplementary information of \cite{Priebe:17} for the SQUIRRELS algorithm. This means that we generate a realistic density matrix \(\rho^{\dagger}\) of size \(41\times41\) based on pump interaction with coupling strength \(g_{\text{pump}}=1.73\) where we also include a slight phase jitter. 
Noise-free data is computed by applying the forward operator with \(g=3g_{\text{pump}}\) for 100 equally spaced \(\theta\), from which Poisson distributed synthetic data is generated. 
To demonstrate the convergence of our method, we use different intensities to get a range of different noise levels \(\delta=\calS_{\gobs}(T\rho^{\dagger})\). 
We then simply choose \(\alpha\) to be \(\alpha_{0}\sqrt{\delta}\) which guarantees that \(\frac{\delta}{\alpha}\) tends to zero, but is most likely not optimal for any particular noise level. 
We use the squared \(L^{2}\)-norm and FISTA or the Kullback-Leibler divergence and the Chambolle-Pock algorithm as described in \Cref{sec:algorithms}. As a reference solution we simply choose a normalized version of the identity matrix. We run the algorithms until the duality gap is smaller than \(10^{-6}\) for FISTA and \(10^{-5}\) for Chambolle-Pock or after a maximum of 2,000,000 iterations is reached. 
Both values are small enough for our purposes as they are much smaller than the final distances to the true solution without noise. The bound for the Chambolle-Pock algorithm was chosen slightly higher to compensate for its longer running time. Only for the last two data points of Chambolle-Pock the cap on the number of iterations became relevant. In \Cref{fig:PINEM_conv} we can see the results and observe that in both cases the data and the reconstructed density matrices converge as claimed in \Cref{the:regularity}.
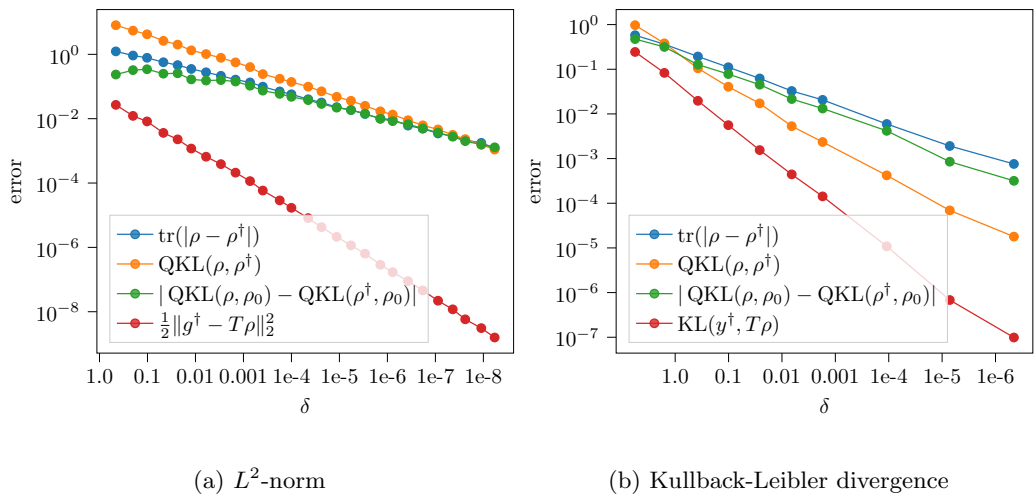
\begin{figure}
\begin{subfigure}[b]{0.45\textwidth}
    \centering
    \input{figures/PINEM_L2_far.tex}
    \caption{\(L^{2}\)-norm}
\end{subfigure}
\begin{subfigure}[b]{0.45\textwidth}
    \centering
    \input{figures/PINEM_KL_far.tex}
    \caption{Kullback-Leibler divergence}
\end{subfigure}
\caption{Convergence behaviour for different data fidelity functionals}
\label{fig:PINEM_conv}
\end{figure}

\subsection{Homodyne tomography}
Homodyne tomography is a well established technique to reconstruct the quantum state of light by using a balanced homodyne detector. It was developed by Vogel and Risken in \cite{Vogel:89} and later demonstrated experimentally in \cite{Smithey:93}. The measured data can be viewed as the Radon transform of the Wigner function, which is related to the density matrix via the Fourier transform. 
Therefore, originally methods from classical tomography for the inversion of the Radon transform, like filtered back-projection, have been used to reconstruct quantum states from experimental data \cite{Leonhardt:95}. 
Another approach for the inversion has been the derivation of pattern functions which can be multiplied with the data and then integrated to directly obtain any element of the density matrix \cite{Richter:00}. To guarantee physically meaningful results, it is also possible to use a fixed point algorithm based on maximum likelihood estimation \cite{Rehacek:07}. Furthermore, methods from convex optimization have recently been used to incorporate the trace and positive semi-definiteness constraints \cite{Strandberg:22}.

We will consider homodyne tomography with perfect detection efficiency 
for which the forward operator \(T\) can be defined by
\begin{align}\label{eq:T_homodyne}
    (T\rho)(\theta,x)=\sum_{m,n\in\setN}\rho_{m,n}\euler^{\iu (n-m)\theta}u_{m}(x)u_{n}(x)
\qquad    \text{with }u_{m}(x)=\frac{1}{\sqrt{\sqrt{\pi}m!2^{m}}}H_{m}(x)\euler^{-\frac{x^2}{2}}.
\end{align}
Here \(H_{m}\) denotes the \(m\)-th Hermite polynomial \cite{Richter:00}. Note that the functions \(u_{m}\euler^{\iu m\theta}\) form an orthonormal basis of \(L^{2}(\setR,\setC)\) and we can define \(U_{\theta}:\ell^{2}(\setN,\setC)\rightarrow L^{2}(\setR,\setC)\), \(e_{n}\mapsto \euler^{\iu n\theta}u_{n}\) with the standard basis 
\(\{e_n:n\in\setN\}\) of \(\ell^{2}(\setN,\setC)\).

Using a Gelfand triple and the Dirac delta function one could then formally write
\begin{align*}
    (T\rho)(\theta,x)=\langle\delta_{x},U_{\theta}\rho U_{\theta}^{*}\delta_{x}\rangle
\end{align*}
to get to a similar structure as we have assumed until now.
Unfortunately, the forward operator does not have the required continuity properties as the 
following lemma shows: 
\begin{lemma}
    The operator \(T\) in \eqref{eq:T_homodyne} is not weak-\(*\) to weak-\(*\) sequentially continuous from 
    $\BopI$ to $L^2$. 
\end{lemma}
\begin{proof}
    We just show this for a single fixed \(\theta\), but the argument can easily be extended to an interval
    of $\theta$'s. Take a bounded real-valued function
    \(\tilde{a}\in (L^{2}(\setR)\setminus\{0\})\). 
    Then the multiplication operator \(M_{\tilde{a}}:L^{2}(\setR)\rightarrow L^{2}(\setR)\) is bounded and self-adjoint, but not compact as it is not identically \(0\). Then the operator \(U_{\theta}^{*}M_{\tilde{a}}U_{\theta}\) is also not compact.  

    For a linear functional \(F\) on \(\BopI\) weak-\(*\) continuity is equivalent to both weak-\(*\) continuity  on the unit ball and the existence of an element \(K\in\Kop\) such that \(F=\tr(K\cdot)\) \cite[Corollary~4.46]{Fabian:01}. Therefore, the linear functional \(\rho \mapsto \tr(U_{\theta}^{*}M_{\tilde{a}}U_{\theta}\rho)\) cannot be weak-\(*\) continuous and it also cannot be weak-\(*\) continuous on the unit ball. As \(\Kop\) is separable, the unit ball of \(\BopI\) with the weak-\(*\) topology is metrizable and the functional is also not sequentially weak-\(*\) continuous \cite[Proposition~3.24]{Fabian:01}. This implies that there exists a weak-\(*\) convergent sequence \((\rho^{(k)})_{k\in\setN}\subset\BopI\) converging to \(\rho\) such that \(\tr(U_{\theta}^{*}M_{\tilde{a}}U_{\theta}\rho^{(k)})\) does not converge to \(\tr(U_{\theta}^{*}M_{\tilde{a}}U_{\theta}\rho)\). Now we observe
    \begin{align*}
        \tr(U_{\theta}^{*}M_{\tilde{a}}U_{\theta}\rho)&=\sum_{m,n}\langle e_{m},U_{\theta}^{*}M_{\tilde{a}}U_{\theta}e_{n}\rangle\langle e_{m},\rho e_{n}\rangle=\sum_{m,n}\langle \euler^{\iu m\theta}u_{m} ,M_{\tilde{a}}\euler^{\iu n\theta}u_{n}\rangle\rho_{m,n}\\
        &=\sum_{m,n}\int_{\setR}\tilde{a}(x) \euler^{-\iu m\theta}u_{m}(x)\euler^{\iu n\theta}u_{n}(x)\dd x\rho_{m,n}\\
        &=\int_{\setR}\tilde{a}(x) \sum_{m,n}\rho_{m,n}\euler^{-\iu (n-m)\theta}u_{m}(x)u_{n}(x)\dd x=\langle\tilde{a},(T\rho)(\theta)\rangle_{L^{2}(\setR)}.
    \end{align*}
    So \((T\rho^{(k)})(\theta)\) does not converge weak-\(*\) to \((T\rho)(\theta)\).
\end{proof}
\begin{remark}
    This problem also cannot be fixed by considering a semi-discrete version and switching to a sequence space by defining \(T_{\mathrm{semi}}:\tilde{\BopI}\rightarrow L^{2}([0,\pi),\ell^{1}(\setZ))\)
    \begin{align*}
        (T_{\mathrm{semi}}\rho)(\theta,l):=\int_{lh}^{(l+1)h}(T\rho)(\theta,x)\dd x
    \end{align*}
    for some $h>0$.
    The proof for this is essentially the same, but now instead of \(\tilde{a}\in L^{2}(\setR)\) one considers for a sequence \(a\in c_{0}(\setZ)\) the function \(\tilde{a}(x):=\sum_{l}a_{l}1_{[lh,(l+1)h)}(x)\), which also does not generate a compact multiplication operator and satisfies 
    \(\langle\tilde{a},(T\rho)(\theta)\rangle_{L^{2}(\setR)} = h\langle a, (T_{\mathrm{semi}}\rho)(\theta)\rangle_{\ell^1(\setZ)}\).
\end{remark}
\begin{remark}
    To the best of our knowledge, even the question whether the operator is continuous with respect to the norm topologies appears to be still open. Its relationship to the Radon transform \cite{Vogel:89} on \(L^{2}(\setR^{2})\) suggests, that on \(\BopII\) it is injective \cite{Ludwig:66} but unbounded 
    with respect to $L^2$ as codomain  \cite{Solmon:76}. 
\end{remark}
One way to still apply our theory, is to use decay conditions on the matrix elements $\rho_{m,n}$ 
and apply \Cref{lemm:embedding}.  
E.g., similar to \cite{Aubry:08}, we could assume that $\rho$ belongs to a weighted $\BopII$ space 
with norm $\|\rho\|_{\hat{\calB}}:=\|(\rho_{m,n}\exp(B(m+n)^{r/2}))_{m,n}\|_{\BopII}$ for some $B>0$ and $r\in (0,2)$
and define $\dom T$ as a closed ball in $\hat{\calB}$. As $\hat{\calB}$ is a separable Hilbert space,  
$\dom T$ is weakly sequentially compact, 
and from $\|\rho\|_{\BopI} = \sup\{|\tr(\sigma\rho)|:\|\sigma\|_{\Bop}\leq 1,\,\sigma\in\Kop\}
\leq \sum_{m,n}|\rho_{m,n}|$ it follows from the Cauchy-Schwarz inequality that $\hat{\calB}$ is continuously embedded in $\BopI$. 
To prove continuity of $T$ in \eqref{eq:T_homodyne}, we can proceed as in \cite{Aubry:08} and 
use the fact that $T\rho$ is the Radon transform of the Wigner function $W_{\rho}$ which
satisfies $\|W_{\rho}\|_{L^2(\setR^2)}^2=\|\rho\|_{\BopII}^2/(2\pi)$ \cite[eq.(5)]{Aubry:08} and 
exhibits an exponential decay controlled by $\sup_{m,n}||\rho_{m,n}|\exp(B'(m+n)^{r/2})$ 
\cite[Prop. 1]{Aubry:08}. 
For $B'<B$ the latter norm can be bounded by $\|\rho\|_{\hat{\calB}}$ using the Cauchy-Schwarz inequality. 

In principle, a uniform hard constraint on $\|\rho\|_{\hat{\calB}}$ for a sequence of data with vanishing noise level would have to be implemented numerically to guarantee convergence of the corresponding estimators, which 
is numerically inconvenient. 

A different approach to obtaining a well-defined operator within our framework avoiding strong a-priori assumptions
on $\rho$,  is to replace the Dirac delta functions by any orthonormal basis \(\{v_{l},l\in\setZ\}\) of \(L^{2}(\setR,\setC)\) and get \(\tilde{T}:\tilde{\BopI}\rightarrow L^2([0,\pi),\ell^{1}(\setZ))\) by
\[(\tilde{T}\rho)(\theta,l)=\langle v_{l},U_{\theta}\rho U_{\theta}^{*}v_{l}\rangle.
\]
Note that this involves integration over squares of the Schwartz kernel of $U_{\theta}\rho U_{\theta}^{*}$ rather than integration only over the diagonal. 
To still fit the measurement data, the choice of \(v_{l}\) should be well localized. One could for example use an appropriately scaled Haar wavelet basis and in the end arrive at an operator which is close to \(T_{\mathrm{semi}}\), particularly when passing to the fully discrete setting with sufficiently small bin size.

After these modification we can easily see, that \(\tilde{T}\) is weak-\(*\) to weak-\(*\) continuous by \Cref{lem:cont_op}. The injectivity of  \(\tilde{T}\) depends on the choice of the actual basis \(v_{l}\) but as we stated before also in the non-injective case we still get the results from \Cref{the:regularity} at least for a subsequence.

As an example we consider the reconstruction of a Schrödinger's cat state with \(a=3\) which is the superposition of the two coherent states \(|a\rangle+|-a\rangle\). It is a standard example for a state that exhibits non-classical behavior \cite{Leonhardt:95}. 
We consider a density matrix of size \(21\times21\) and use \(60\) different phases and choose \(v_{l}\) to be indicator functions of the intervals \([-5+\frac{l}{12},-5+\frac{l+1}{12}]\) for \(0\leq l<120\) such that we cover the interval \([-5,5]\) which contains the relevant part of the data. 
We generate Poisson distributed synthetic data using the operator \(T_{\mathrm{semi}}\). The stopping rule, \(\delta\) and \(\alpha\) are then chosen in the same way as in the PINEM example. The restriction on the number of iterations is again only relevant for the last three data points with the Kullback-Leibler divergence as a data fidelity functional. 
We perform the reconstruction with different data fidelity functionals and we use either \(T_{\mathrm{semi}}\) or \(\tilde{T}\) and compare the results. In \Cref{fig:Homodyne_conv} we can see that if the noise is not too small, both operators lead to very similar results. 
However, as expected, at some point their difference becomes significant and only using \(T_{\mathrm{semi}}\) still leads to further convergence. 

\begin{figure}
\begin{subfigure}[b]{0.45\textwidth}
    \centering
    \input{figures/Homodyne_L2_far.tex}
    \caption{\(L^{2}\)-norm}
\end{subfigure}
\begin{subfigure}[b]{0.45\textwidth}
    \centering
    \input{figures/Homodyne_KL_far.tex}
    \caption{Kullback-Leibler divergence}
\end{subfigure}
\caption{Convergence behavior for different data fidelity functionals. Dotted lines indicate results for the exact forward operator.}
\label{fig:Homodyne_conv}
\end{figure}
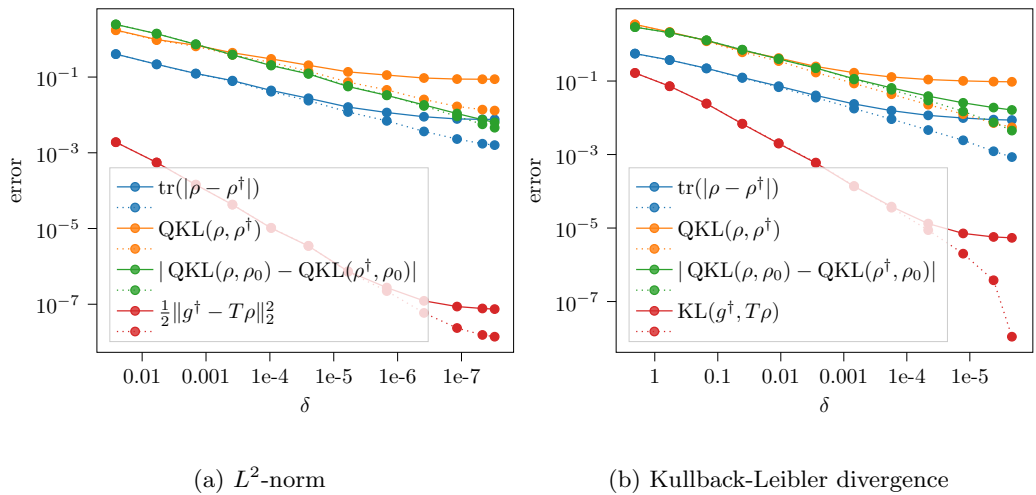

%% file: figures/PINEM_L2_far.tex
\begin{tikzpicture}[scale=0.8]

\definecolor{crimson2143940}{RGB}{214,39,40}
\definecolor{darkgray176}{RGB}{176,176,176}
\definecolor{darkorange25512714}{RGB}{255,127,14}
\definecolor{forestgreen4416044}{RGB}{44,160,44}
\definecolor{lightgray204}{RGB}{204,204,204}
\definecolor{steelblue31119180}{RGB}{31,119,180}

\begin{axis}[
legend cell align={left},
legend style={
  fill opacity=0.8,
  draw opacity=1,
  text opacity=1,
  at={(0.03,0.03)},
  anchor=south west,
  draw=lightgray204
},
log basis x={10},
log basis y={10},
tick align=outside,
tick pos=left,
x grid style={darkgray176},
xlabel={\(\displaystyle \delta\)},
xmin=0.88266905956053, xmax=426322593.912704,
xmode=log,
xtick style={color=black},
xtick={1,10,100,1000,10000,100000,1000000,10000000,100000000},
xticklabels={1.0,0.1,0.01,0.001,1e-4,1e-5,1e-6,1e-7,1e-8},
y grid style={darkgray176},
ylabel={error},
ymin=5.17078248827539e-10, ymax=24.6500402614802,
ymode=log,
ytick style={color=black},
ytick={1e-12,1e-10,1e-08,1e-06,0.0001,0.01,1,100,10000},
yticklabels={
  \(\displaystyle {10^{-12}}\),
  \(\displaystyle {10^{-10}}\),
  \(\displaystyle {10^{-8}}\),
  \(\displaystyle {10^{-6}}\),
  \(\displaystyle {10^{-4}}\),
  \(\displaystyle {10^{-2}}\),
  \(\displaystyle {10^{0}}\),
  \(\displaystyle {10^{2}}\),
  \(\displaystyle {10^{4}}\)
}
]
\addplot [semithick, steelblue31119180, mark=*, mark size=2, mark options={solid}]
table {%
2.19039514989327 1.23628791562812
4.98879045221715 0.918366691403635
9.79033615991703 0.781443631666918
21.4478233700518 0.568317698967709
42.4678528549667 0.462759847123509
81.7264055457022 0.351863558659117
166.24980827564 0.277344009717243
339.142417028212 0.215184535440002
683.208249277772 0.163889201109207
1371.80340618329 0.134022417807741
2529.1395945367 0.0990160505334247
5697.80647661297 0.070225740750049
9981.12525494408 0.0574002917805102
22325.6466220019 0.0402634045536201
42311.2107265164 0.0318087555092146
87659.2806228682 0.0227890856910244
174117.984744357 0.0185941099300559
341918.519333409 0.0139804633180976
712749.824289369 0.0101077706661804
1285914.72795044 0.00932042269963306
2686807.38627157 0.00609751146394149
5459487.60683369 0.00489033361110336
11350565.2633126 0.00347478722348898
22966541.0557599 0.00285874534663819
41870713.5722433 0.00216522394581941
90466303.4469596 0.00177627171875646
171796291.211048 0.00122668106935575
};
\addlegendentry{\(\tr(|\rho-\rho^{\dagger}|)\)}
\addplot [semithick, darkorange25512714, mark=*, mark size=2, mark options={solid}]
table {%
2.19039514989327 8.06198742534603
4.98879045221715 5.4622382944007
9.79033615991703 4.20657351244423
21.4478233700518 2.62603744730501
42.4678528549667 2.01024482435521
81.7264055457022 1.33160516940855
166.24980827564 1.02244671265517
339.142417028212 0.780482475784809
683.208249277772 0.56344919286965
1371.80340618329 0.402915627552368
2529.1395945367 0.243020951845962
5697.80647661297 0.17453807012231
9981.12525494408 0.138646315659075
22325.6466220019 0.0989702013587797
42311.2107265164 0.0710399157182044
87659.2806228682 0.0478053282358133
174117.984744357 0.0359784557654905
341918.519333409 0.0251243483594998
712749.824289369 0.0169414807731412
1285914.72795044 0.0134060589686726
2686807.38627157 0.00896903050885434
5459487.60683369 0.00622434969589059
11350565.2633126 0.0046304525110884
22966541.0557599 0.00314439385734167
41870713.5722433 0.00232291392043971
90466303.4469596 0.00156218389590179
171796291.211048 0.00109582364076743
};
\addlegendentry{\(\QKL(\rho,\rho^{\dagger})\)}
\addplot [semithick, forestgreen4416044, mark=*, mark size=2, mark options={solid}]
table {%
2.19039514989327 0.235637641046254
4.98879045221715 0.323445909257861
9.79033615991703 0.344429910123456
21.4478233700518 0.250745942940003
42.4678528549667 0.257816789480849
81.7264055457022 0.166499925903548
166.24980827564 0.153288435863059
339.142417028212 0.158548753459289
683.208249277772 0.142633035944009
1371.80340618329 0.105792148514665
2529.1395945367 0.0745644954789908
5697.80647661297 0.0590097347822556
9981.12525494408 0.0478795663233575
22325.6466220019 0.0377643870251969
42311.2107265164 0.0287613186932862
87659.2806228682 0.0220633008522606
174117.984744357 0.018634128170246
341918.519333409 0.0143378609305436
712749.824289369 0.00990357776039907
1285914.72795044 0.00836366854393056
2686807.38627157 0.00672362356179956
5459487.60683369 0.00499666906433571
11350565.2633126 0.00368667685299773
22966541.0557599 0.00275385279104778
41870713.5722433 0.00200559871553407
90466303.4469596 0.0015600673189895
171796291.211048 0.00128769030128373
};
\addlegendentry{\(|\QKL(\rho,\rho_{0})-\QKL(\rho^{\dagger},\rho_{0})|\)}
\addplot [semithick, crimson2143940, mark=*, mark size=2, mark options={solid}]
table {%
2.19039514989327 0.0268818348279181
4.98879045221715 0.0122213445421725
9.79033615991703 0.00822762294914741
21.4478233700518 0.00362130975566543
42.4678528549667 0.00227791031369707
81.7264055457022 0.00118083873188486
166.24980827564 0.000657809829725585
339.142417028212 0.000386898728937549
683.208249277772 0.000210139293191879
1371.80340618329 0.000114889714525362
2529.1395945367 5.84495807934328e-05
5697.80647661297 2.88611617881446e-05
9981.12525494408 1.69509057865512e-05
22325.6466220019 8.12148787216145e-06
42311.2107265164 4.2585686515571e-06
87659.2806228682 2.12694777705431e-06
174117.984744357 1.15354093148001e-06
341918.519333409 6.38411812520186e-07
712749.824289369 2.87450249690689e-07
1285914.72795044 1.69419238236696e-07
2686807.38627157 8.94948176530247e-08
5459487.60683369 4.59456032951379e-08
11350565.2633126 2.19775182935217e-08
22966541.0557599 1.18621948613515e-08
41870713.5722433 5.80598502601433e-09
90466303.4469596 3.09364275803685e-09
171796291.211048 1.58099969392937e-09
};
\addlegendentry{\(\frac{1}{2}\|g^{\dagger}-T\rho\|^{2}_{2}\)}
\end{axis}

\end{tikzpicture}

%% file: figures/PINEM_KL_far.tex
\begin{tikzpicture}[scale=0.8]

\definecolor{crimson2143940}{RGB}{214,39,40}
\definecolor{darkgray176}{RGB}{176,176,176}
\definecolor{darkorange25512714}{RGB}{255,127,14}
\definecolor{forestgreen4416044}{RGB}{44,160,44}
\definecolor{lightgray204}{RGB}{204,204,204}
\definecolor{steelblue31119180}{RGB}{31,119,180}

\begin{axis}[
legend cell align={left},
legend style={
  fill opacity=0.8,
  draw opacity=1,
  text opacity=1,
  at={(0.03,0.03)},
  anchor=south west,
  draw=lightgray204
},
log basis x={10},
log basis y={10},
tick align=outside,
tick pos=left,
x grid style={darkgray176},
xlabel={\(\displaystyle \delta\)},
xmin=0.0784020515727455, xmax=5004629.3799217,
xmode=log,
xtick style={color=black},
xtick={1,10,100,1000,10000,100000,1000000},
xticklabels={1.0,0.1,0.01,0.001,1e-4,1e-5,1e-6},
y grid style={darkgray176},
ylabel={error},
ymin=4.41348907740532e-08, ymax=2.1690926274976,
ymode=log,
ytick style={color=black},
ytick={1e-09,1e-08,1e-07,1e-06,1e-05,0.0001,0.001,0.01,0.1,1,10,100},
yticklabels={
  \(\displaystyle {10^{-9}}\),
  \(\displaystyle {10^{-8}}\),
  \(\displaystyle {10^{-7}}\),
  \(\displaystyle {10^{-6}}\),
  \(\displaystyle {10^{-5}}\),
  \(\displaystyle {10^{-4}}\),
  \(\displaystyle {10^{-3}}\),
  \(\displaystyle {10^{-2}}\),
  \(\displaystyle {10^{-1}}\),
  \(\displaystyle {10^{0}}\),
  \(\displaystyle {10^{1}}\),
  \(\displaystyle {10^{2}}\)
}
]
\addplot [semithick, steelblue31119180, mark=*, mark size=2, mark options={solid}]
table {%
0.177460694529429 0.577629485844541
0.627372637363034 0.363981267832963
2.66373372372582 0.193071786706016
9.88526926316423 0.111030238465943
38.2864353773458 0.0622691023831853
152.279702027429 0.0326296877827974
577.226848469127 0.0206364560671324
9178.27714122686 0.00594873505512874
139034.630640046 0.00191438985098541
2211042.91171378 0.000758241282494246
};
\addlegendentry{\(\tr(|\rho-\rho^{\dagger}|)\)}
\addplot [semithick, darkorange25512714, mark=*, mark size=2, mark options={solid}]
table {%
0.177460694529429 0.9697608376021
0.627372637363034 0.380533006258393
2.66373372372582 0.105030993006434
9.88526926316423 0.0405732200611435
38.2864353773458 0.0172208954623139
152.279702027429 0.00532048806403695
577.226848469127 0.00235297345635832
9178.27714122686 0.000421058213557535
139034.630640046 6.91658628142466e-05
2211042.91171378 1.79163244231617e-05
};
\addlegendentry{\(\QKL(\rho,\rho^{\dagger})\)}
\addplot [semithick, forestgreen4416044, mark=*, mark size=2, mark options={solid}]
table {%
0.177460694529429 0.47389843072941
0.627372637363034 0.316045101571781
2.66373372372582 0.126789643514799
9.88526926316423 0.0780690703725746
38.2864353773458 0.045175261589196
152.279702027429 0.0215525587627128
577.226848469127 0.0132128064942241
9178.27714122686 0.00417840150495152
139034.630640046 0.000851788310924473
2211042.91171378 0.000317653859727152
};
\addlegendentry{\(|\QKL(\rho,\rho_{0})-\QKL(\rho^{\dagger},\rho_{0})|\)}
\addplot [semithick, crimson2143940, mark=*, mark size=2, mark options={solid}]
table {%
0.177460694529429 0.2440984407769
0.627372637363034 0.0829056826493324
2.66373372372582 0.0196508188086152
9.88526926316423 0.00560436954899948
38.2864353773458 0.00154932162715486
152.279702027429 0.000443423635034325
577.226848469127 0.000142480614437429
9178.27714122686 1.0855766600369e-05
139034.630640046 6.81296759215553e-07
2211042.91171378 9.87178100841091e-08
};
\addlegendentry{\(\KL(y^{\dagger},T\rho)\)}
\end{axis}

\end{tikzpicture}

%% file: figures/Homodyne_L2_far.tex
\begin{tikzpicture}[scale=0.8]

\definecolor{crimson2143940}{RGB}{214,39,40}
\definecolor{darkgray176}{RGB}{176,176,176}
\definecolor{darkorange25512714}{RGB}{255,127,14}
\definecolor{forestgreen4416044}{RGB}{44,160,44}
\definecolor{lightgray204}{RGB}{204,204,204}
\definecolor{steelblue31119180}{RGB}{31,119,180}

\begin{axis}[
legend cell align={left},
legend style={
  fill opacity=0.8,
  draw opacity=1,
  text opacity=1,
  at={(0.03,0.03)},
  anchor=south west,
  draw=lightgray204
},
log basis x={10},
log basis y={10},
tick align=outside,
tick pos=left,
x grid style={darkgray176},
xlabel={\(\displaystyle \delta\)},
xmin=19.5730750303187, xmax=65447613.3299811,
xmode=log,
xtick style={color=black},
xtick={100,1000,10000,100000,1000000,10000000},
xticklabels={0.01,0.001,1e-4,1e-5,1e-6,1e-7},
y grid style={darkgray176},
ylabel={error},
ymin=5.37160182264689e-09, ymax=6.35636517281559,
ymode=log,
ytick style={color=black},
ytick={1e-11,1e-09,1e-07,1e-05,0.001,0.1,10,1000},
yticklabels={
  \(\displaystyle {10^{-11}}\),
  \(\displaystyle {10^{-9}}\),
  \(\displaystyle {10^{-7}}\),
  \(\displaystyle {10^{-5}}\),
  \(\displaystyle {10^{-3}}\),
  \(\displaystyle {10^{-1}}\),
  \(\displaystyle {10^{1}}\),
  \(\displaystyle {10^{3}}\)
}
]
\addplot [semithick, steelblue31119180, mark=*, mark size=2, mark options={solid}]
table {%
38.7449572495888 0.404488641398678
168.604133574234 0.217313394063999
691.684258658193 0.12430259732215
2591.66121944345 0.0804091616374034
10464.4836687182 0.0443092542978196
40170.5043287831 0.0275383266456111
167318.077634455 0.0161317879253513
676580.327522761 0.011569164539409
2585393.43498426 0.0089514848813677
8513768.80595343 0.00783344434773956
21268047.4228038 0.00764481902910348
33062652.2055745 0.00758439418186847
};
\addlegendentry{\(\tr(|\rho-\rho^{\dagger}|)\)}
\addplot [semithick, steelblue31119180, dotted, mark=*, mark size=2, mark options={solid}]
table {%
38.7449572495888 0.403370140629862
168.604133574234 0.21601608604805
691.684258658193 0.122517121034217
2591.66121944345 0.0778748223752076
10464.4836687182 0.0410432174504019
40170.5043287831 0.0238247471333078
167318.077634455 0.0119552096030684
676580.327522761 0.00697044806239015
2585393.43498426 0.00366278082750012
8513768.80595343 0.00230828363413521
21268047.4228038 0.00173564010067182
33062652.2055745 0.00159425611087924
};
\addlegendentry{}
\addplot [semithick, darkorange25512714, mark=*, mark size=2, mark options={solid}]
table {%
38.7449572495888 1.74438235723073
168.604133574234 0.979071507977843
691.684258658193 0.69284327889831
2591.66121944345 0.438639612869815
10464.4836687182 0.301813132501362
40170.5043287831 0.205063861983544
167318.077634455 0.137079308989649
676580.327522761 0.112898714440321
2585393.43498426 0.0943061726763979
8513768.80595343 0.0884987579176797
21268047.4228038 0.0876458728213007
33062652.2055745 0.0878787137351928
};
\addlegendentry{\(\QKL(\rho,\rho^{\dagger})\)}
\addplot [semithick, darkorange25512714, dotted, mark=*, mark size=2, mark options={solid}]
table {%
38.7449572495888 1.72045783753107
168.604133574234 0.944393003794099
691.684258658193 0.646927601423789
2591.66121944345 0.385913779367157
10464.4836687182 0.243318913093404
40170.5043287831 0.143336515713395
167318.077634455 0.0736869702443317
676580.327522761 0.0458533942425858
2585393.43498426 0.0256050527798027
8513768.80595343 0.0167669855092332
21268047.4228038 0.0137972306288158
33062652.2055745 0.0130313869702134
};
\addlegendentry{}
\addplot [semithick, forestgreen4416044, mark=*, mark size=2, mark options={solid}]
table {%
38.7449572495888 2.45544986452759
168.604133574234 1.39609625612907
691.684258658193 0.729659560409808
2591.66121944345 0.385492254535126
10464.4836687182 0.20138100182114
40170.5043287831 0.120875926839242
167318.077634455 0.0563121501634427
676580.327522761 0.0328661328429867
2585393.43498426 0.018362834209861
8513768.80595343 0.0108677099652104
21268047.4228038 0.00747450311362474
33062652.2055745 0.00638728441215264
};
\addlegendentry{\(|\QKL(\rho,\rho_{0})-\QKL(\rho^{\dagger},\rho_{0})|\)}
\addplot [semithick, forestgreen4416044, dotted, mark=*, mark size=2, mark options={solid}]
table {%
38.7449572495888 2.45920636710743
168.604133574234 1.40320129362842
691.684258658193 0.737857266123362
2591.66121944345 0.392959679428161
10464.4836687182 0.207494348598801
40170.5043287831 0.125256165371182
167318.077634455 0.0584765194903021
676580.327522761 0.0333608168546418
2585393.43498426 0.0172495098105525
8513768.80595343 0.00916220632122577
21268047.4228038 0.00565289526378621
33062652.2055745 0.0045959719836306
};
\addlegendentry{}
\addplot [semithick, crimson2143940, mark=*, mark size=2, mark options={solid}]
table {%
38.7449572495888 0.00190651364488668
168.604133574234 0.000553770550971346
691.684258658193 0.000143412608699797
2591.66121944345 4.25435779874154e-05
10464.4836687182 1.028211665011e-05
40170.5043287831 3.45125138112771e-06
167318.077634455 7.40766499838173e-07
676580.327522761 2.75268539278827e-07
2585393.43498426 1.22914502873811e-07
8513768.80595343 8.68831650442453e-08
21268047.4228038 7.73126697054986e-08
33062652.2055745 7.44263650299914e-08
};
\addlegendentry{\(\frac{1}{2}\|g^{\dagger}-T\rho\|^{2}_{2}\)}
\addplot [semithick, crimson2143940, dotted, mark=*, mark size=2, mark options={solid}]
table {%
38.7449572495888 0.00190210514992676
168.604133574234 0.00055354473122818
691.684258658193 0.000143957160482793
2591.66121944345 4.29025232359092e-05
10464.4836687182 1.0458305223583e-05
40170.5043287831 3.50091206311468e-06
167318.077634455 7.1552700503347e-07
676580.327522761 2.23004935612953e-07
2585393.43498426 5.90077674634579e-08
8513768.80595343 2.35542423113393e-08
21268047.4228038 1.54623000611821e-08
33062652.2055745 1.38840982214381e-08
};
\addlegendentry{}
\end{axis}

\end{tikzpicture}

%% file: figures/Homodyne_KL_far.tex
\begin{tikzpicture}[scale=0.8]

\definecolor{crimson2143940}{RGB}{214,39,40}
\definecolor{darkgray176}{RGB}{176,176,176}
\definecolor{darkorange25512714}{RGB}{255,127,14}
\definecolor{forestgreen4416044}{RGB}{44,160,44}
\definecolor{lightgray204}{RGB}{204,204,204}
\definecolor{steelblue31119180}{RGB}{31,119,180}

\begin{axis}[
legend cell align={left},
legend style={
  fill opacity=0.8,
  draw opacity=1,
  text opacity=1,
  at={(0.03,0.03)},
  anchor=south west,
  draw=lightgray204
},
log basis x={10},
log basis y={10},
tick align=outside,
tick pos=left,
x grid style={darkgray176},
xlabel={\(\displaystyle \delta\)},
xmin=0.24372501156914, xmax=1000000,
xmode=log,
xtick style={color=black},
xtick={1,10,100,1000,10000,100000},
xticklabels={1,0.1,0.01,0.001,1e-4,1e-5},
y grid style={darkgray176},
ylabel={error},
ymin=4.15204652179783e-09, ymax=9.24799257056261,
ymode=log,
ytick style={color=black},
ytick={1e-11,1e-09,1e-07,1e-05,0.001,0.1,10,1000},
yticklabels={
  \(\displaystyle {10^{-11}}\),
  \(\displaystyle {10^{-9}}\),
  \(\displaystyle {10^{-7}}\),
  \(\displaystyle {10^{-5}}\),
  \(\displaystyle {10^{-3}}\),
  \(\displaystyle {10^{-1}}\),
  \(\displaystyle {10^{1}}\),
  \(\displaystyle {10^{3}}\)
}
]
\addplot [semithick, steelblue31119180, mark=*, mark size=2, mark options={solid}]
table {%
0.485263333160379 0.558108870448287
1.72317446233225 0.373158532798775
6.61146592276818 0.222216479850829
24.4682104533125 0.124786042721041
92.5023164568258 0.0725150005333798
359.598149012088 0.0405865258042499
1454.13487649823 0.0237980457198338
5710.44015262327 0.0156322106147424
21847.0839803997 0.0116481057265173
78375.4823025333 0.00984718740258033
238120.896649755 0.00891579725559884
465077.171550881 0.00867204617763544
};
\addlegendentry{\(\tr(|\rho-\rho^{\dagger}|)\)}
\addplot [semithick, steelblue31119180, dotted, mark=*, mark size=2, mark options={solid}]
table {%
0.485263333160379 0.557413823002402
1.72317446233225 0.371775055167356
6.61146592276818 0.219762102901776
24.4682104533125 0.121472766869419
92.5023164568258 0.0680058356144126
359.598149012088 0.0353951567783034
1454.13487649823 0.0178092608397489
5710.44015262327 0.00918065941950423
21847.0839803997 0.00463793337343307
78375.4823025333 0.00244011372019866
238120.896649755 0.00122829111929858
465077.171550881 0.000853442447860589
};
\addlegendentry{}
\addplot [semithick, darkorange25512714, mark=*, mark size=2, mark options={solid}]
table {%
0.485263333160379 3.47654729673242
1.72317446233225 2.15962037404373
6.61146592276818 1.2605457060607
24.4682104533125 0.666549708499212
92.5023164568258 0.422300192618392
359.598149012088 0.249314122504751
1454.13487649823 0.168282744950093
5710.44015262327 0.128689009838652
21847.0839803997 0.109750669424822
78375.4823025333 0.101027208663299
238120.896649755 0.0964386029520392
465077.171550881 0.0953102615851482
};
\addlegendentry{\(\QKL(\rho,\rho^{\dagger})\)}
\addplot [semithick, darkorange25512714, dotted, mark=*, mark size=2, mark options={solid}]
table {%
0.485263333160379 3.45205302340579
1.72317446233225 2.12414260982294
6.61146592276818 1.20906478644564
24.4682104533125 0.605617383278535
92.5023164568258 0.34960817121582
359.598149012088 0.170661262991048
1454.13487649823 0.0858257809825796
5710.44015262327 0.0437843254061965
21847.0839803997 0.0227126716224337
78375.4823025333 0.0125661857323971
238120.896649755 0.00720420519777202
465077.171550881 0.00550343846710498
};
\addlegendentry{}
\addplot [semithick, forestgreen4416044, mark=*, mark size=2, mark options={solid}]
table {%
0.485263333160379 2.92829102826097
1.72317446233225 2.0676664815382
6.61146592276818 1.27815973035089
24.4682104533125 0.708105158340611
92.5023164568258 0.397755251710605
359.598149012088 0.226131260555158
1454.13487649823 0.115110123863476
5710.44015262327 0.0649857316026292
21847.0839803997 0.0388796970013106
78375.4823025333 0.0255329553330403
238120.896649755 0.0189648099632729
465077.171550881 0.0163211561646497
};
\addlegendentry{\(|\QKL(\rho,\rho_{0})-\QKL(\rho^{\dagger},\rho_{0})|\)}
\addplot [semithick, forestgreen4416044, dotted, mark=*, mark size=2, mark options={solid}]
table {%
0.485263333160379 2.93390719740806
1.72317446233225 2.07487725614965
6.61146592276818 1.28536800507698
24.4682104533125 0.713790670458272
92.5023164568258 0.400446846376981
359.598149012088 0.225184304336195
1454.13487649823 0.110865904460647
5710.44015262327 0.0578759066741092
21847.0839803997 0.029695588992241
78375.4823025333 0.0149035811293636
238120.896649755 0.00748196850943117
465077.171550881 0.00448235277834375
};
\addlegendentry{}
\addplot [semithick, crimson2143940, mark=*, mark size=2, mark options={solid}]
table {%
0.485263333160379 0.166028215053251
1.72317446233225 0.072238346099957
6.61146592276818 0.0242491785592536
24.4682104533125 0.00683729830829763
92.5023164568258 0.00198820974914399
359.598149012088 0.00058899691598014
1454.13487649823 0.000136405753481778
5710.44015262327 3.82627141642105e-05
21847.0839803997 1.31965991282431e-05
78375.4823025333 7.13923427647874e-06
238120.896649755 5.72482698478412e-06
465077.171550881 5.4106574920152e-06
};
\addlegendentry{\(\KL(g^{\dagger},T\rho)\)}
\addplot [semithick, crimson2143940, dotted, mark=*, mark size=2, mark options={solid}]
table {%
0.485263333160379 0.165945833962943
1.72317446233225 0.0723413615785122
6.61146592276818 0.0243753411358821
24.4682104533125 0.00692167550956154
92.5023164568258 0.00203109144413441
359.598149012088 0.0006044419496253
1454.13487649823 0.0001393145031823
5710.44015262327 3.57817927690045e-05
21847.0839803997 8.77745543058011e-06
78375.4823025333 2.00773277096829e-06
238120.896649755 3.79529250139241e-07
465077.171550881 1.10448937146078e-08
};
\addlegendentry{}
\end{axis}

\end{tikzpicture}

%% file: conclusions.tex
As a main theoretical result we have demonstrated the regularizing property of the quantum relative entropy functional as a penalty term in variational regularization methods.  
In contrast to some competing methods, it does not depend on the choice of a basis. 
We have further established all properties of the quantum relative entropy 
in finite dimensional settings required for the application of  algorithms from convex optimization. 
All assumptions of our setting have been verified for a forward operator describing PINEM experiments,  
and stability and convergence of the corresponding density matrix estimators as the noise level tends to $0$ was demonstrated numerically.  
For homodyne tomography, the natural exact forward operator does not fit into our framework but a suitable semi-discrete approximation still leads to partial results. In practice using the exact forward operator works well and quantum relative entropy regularization reliably reconstructs the correct solutions.

The present work suggests a number of new and potentially fruitful lines of research:
A natural question concerns statistical consistency for Poisson data as an analog to the regularizing property 
for deterministic noise models. This seems to require more restrictive assumptions on the forward operator 
and/or the set of possible solutions. A further potential line of research concerns rates of convergence 
of quantum relative entropy penalized estimators under decay conditions on the density matrix elements as 
shown for a projection estimator for homodyne tomography in \cite{Aubry:08} and for the SQUIRRELS estimator for PINEM in \cite{Shi:20}. Furthermore, a proper functional analytic setting for homodyne tomography without 
artificial constraints on the set of density matrices would be interest as well as
applications to further tomographic settings such as homodyne tomography with imperfect detection efficiency (see \cite[Section~5.3]{Leonhardt:95}) or heterodyne tomography (see \cite{Richter:98}).

Finally, we emphasize that positive semidefinite matrices and operators arise in many contexts beyond quantum state reconstruction, for instance as covariance operators.
The results of this paper provide a general new regularization framework for inverse problems in which the unknown is a positive semidefinite operator.

%% file: appendix.tex
\begin{theorem}[General duality gap]\label{the:gen_duality_gap}
    Let \(X,Y\) be Banach spaces, \(T\in\calL(X,Y)\) and let \(\calR:X\rightarrow\setR\cup\{\infty\}\) and \(\calS:Y\rightarrow\setR\cup\{\infty\}\) be convex functionals. Furthermore, for \(\alpha>0\) let
    \[f^{\dagger}\in\underset{f}{\argmin}\frac{1}{\alpha}\calS(Tf)+\calR(f)\]
    such that \(\partial\calR(f^{\dagger})\) and \(\partial\calS(Tf^{\dagger})\) are non-empty. Then for all \(f\in X\), \(p\in Y^{*}\) and \(g\in\partial\calR(f^{\dagger})\) the inequality
    \[\Delta_{\calR}^{g}(f,f^{\dagger})\leq\frac{1}{\alpha}\calS(Tf)+\calR(f)+\calR^{*}(T^{*}p)+\frac{1}{\alpha}\calS^{*}(-\alpha p)\]
    holds.
\end{theorem}
\begin{proof}
Define the primal and dual objective functionals by
\begin{align*}
    \calJ(f)&=\frac{1}{\alpha}\calS(Tf)+\calR(f),\qquad 
    \calJ_{*}(p)=\calR^{*}(T^{*}p)+\frac{1}{\alpha}\calS^{*}(-\alpha p)
\end{align*}
and recall the $\calJ(f)\geq -\calJ_*(p)$ for all $f\in X$ and $p\in Y^*$ by weak Rockafellar-Fenchel 
duality, see \cite{Rockafellar:67}. Then for all \(f,p\) it holds by Young's inequality
\begin{align*}
    \calJ(f)+\calJ_{*}(p)&=\frac{1}{\alpha}\calS(Tf)+\calR(f)+\calR^{*}(T^{*}p)+\frac{1}{\alpha}\calS^{*}(-\alpha p)\\
    &\geq \frac{1}{\alpha}\langle Tf,-\alpha p\rangle+\langle f,T^{*}p\rangle=0
\end{align*}
Now assume that \(f^{\dagger}\) is a solution to the optimization problem. Then \[0\in\partial J(f^{\dagger})=\partial \calR(f^{\dagger})+\frac{1}{\alpha}T^{*}\partial\calS(Tf^{\dagger})\]
Then by out assumptions there exists \(g\in\partial\calR(f^{\dagger})\) and \(q\in \partial\calS(Tf^{\dagger})\) such that \(-\frac{1}{\alpha}T^{*}q=g\)
This implies
\begin{align*}
    &\Delta_{\calR}^{g}(f,f^{\dagger})+\frac{1}{\alpha}\Delta_{\calS}^{q}(Tf,Tf^{\dagger})\\
    =&\calR(f)-\calR(f^{\dagger})-\langle g,f-f^{\dagger}\rangle 
    +\frac{1}{\alpha}\left(\calS(Tf)-\calS(Tf^{\dagger})-\langle q, Tf-Tf^{\dagger}\rangle\right)\\
    =&\calR(f)-\calR(f^{\dagger})-\langle g,f-f^{\dagger}\rangle
    +\frac{1}{\alpha}\left(\calS(Tf)-\calS(Tf^{\dagger})\right)- \frac{1}{\alpha}\langle T^{*}q, f-f^{\dagger}\rangle\\
    =&\calJ(f)-\calJ(f^{\dagger}).
\end{align*}
Now as the Bregman divergence is non-negative, we get that for all \(f,p\)
\begin{align*}
    \Delta_{\calR}^{g}(f,f^{\dagger})\leq\Delta_{\calR}^{g}(f,f^{\dagger})+\frac{1}{\alpha}\Delta_{\calS}^{q}(Tf,Tf^{\dagger})\leq\calJ(f)-\calJ(f^{\dagger})
    \leq \calJ(f)+\calJ_{*}(p)
\end{align*}
using weak duality in the last inequality. 
\end{proof}